\documentclass[]{interact}
 
\usepackage{epstopdf}
\usepackage[caption=false]{subfig}
\usepackage{float}
\usepackage{amsmath,amsfonts,amssymb,amsthm}
\usepackage{enumerate}
\usepackage{stmaryrd}
\usepackage[utf8]{inputenc}
\usepackage[justification=centering]{caption} 
\usepackage[T1]{fontenc}
\usepackage[english]{babel}
\usepackage{graphicx}
\usepackage{color}
\usepackage{pgf,tikz,pgfplots}
\usepackage{lmodern}
\usepackage[colorlinks=true]{hyperref}
\hypersetup{urlcolor=blue, citecolor=blue}

\usepackage{array}
\usepackage{array,multirow,makecell}
\setcellgapes{1pt}
\makegapedcells
\newcolumntype{R}[1]{>{\raggedleft\arraybackslash }b{#1}}
\newcolumntype{L}[1]{>{\raggedright\arraybackslash }b{#1}}
\newcolumntype{C}[1]{>{\centering\arraybackslash }b{#1}}

\usepackage[numbers,sort&compress]{natbib}
\bibpunct[, ]{[}{]}{,}{n}{,}{,}
\makeatletter
\def\NAT@def@citea{\def\@citea{\NAT@separator}}
\makeatother

\theoremstyle{plain}

\theoremstyle{definition}

\theoremstyle{remark}

\newtheorem{deff}{Definition}[section]
\newtheorem{prop}[deff]{Proposition}
\newtheorem{thm}[deff]{Theorem}
\newtheorem{lem}[deff]{Lemma}

\newtheorem{rmk}[deff]{Remark}

\newcounter{cst}

\newcommand*\diff{\mathop{}\!\mathrm{d}}
\def\R{\mathbb{R}}

\def\l{\lambda}

\def\d{{\rm d}}
\def\1{{\bf 1}}
\def\be{\begin{equation}}
\def\ee{\end{equation}}

\begin{document}

\articletype{Research Article}

\title{Uniform \textcolor{black}{estimates} for conforming Galerkin method for anisotropic singularly perturbed elliptic problems}

\author{
\name{David Maltese\textsuperscript{a} and Chokri Ogabi\textsuperscript{b}\thanks{\textsuperscript{b} Corresponding author. Email: Chokri.Ogabi@univ-eiffel.fr} }
\affil{\textsuperscript{a,b}LAMA, Univ. Gustave Eiffel, Univ. Paris Est Cr\'eteil, CNRS, F-77454 Marne-la-Vall\'ee, France.}
}

\maketitle

\begin{abstract}
In this article, we study some anisotropic singular perturbations for a class of linear elliptic problems. A uniform \textcolor{black}{estimates} for conforming $Q_1$ finite element method are derived, and some other results of convergence and regularity for the continuous problem are proved.
\end{abstract}

\section*{Introduction}
The numerical study of singular \textcolor{black}{perturbation} problems keeps \textcolor{black}{an important place in numerical analysis}. Consider an elliptic problem $P_{\epsilon}(u_{\epsilon})=0$ where $\epsilon \in (0,1]$ is a small parameter. Let $P_{\epsilon,h}(u_{\epsilon,h})=0$ be the numerical approximation of the continuous problem $P_{\epsilon}$. 
The \textcolor{black}{estimates} obtained by a classical analysis, for instance, by the Céa's lemma, is of the form 
\begin{equation}\label{glob1}
\Vert u_{\epsilon,h}-u_{\epsilon} \Vert_{\Omega} \leq C\frac{h}{\epsilon^{\alpha}},
\end{equation}
where $\Vert \cdot \Vert_{\Omega}$ is some norm on a suitable space on $\Omega$. To ensure a good numerical approximation of the exact solution $u_{\epsilon}$ when $\epsilon$ is very small, then one must take $h$ much smaller than $\epsilon$, which is impractical from the numerical point of view. For some kind of numerical scheme which called asymptotically preserved, that is when we have $\lim_{\epsilon}\lim_{h}u_{\epsilon,h}=\lim_{h}\lim_{\epsilon}u_{\epsilon,h}$, one can obtain $\epsilon-$uniform \textcolor{black}{estimates}. In \cite{Sin} J. Sin gave a simple method to obtain such \textcolor{black}{estimates}. The idea is combining (\ref{glob1}) with \textcolor{black}{another estimate} of the form 
\begin{equation}\label{glob2}
\Vert u_{\epsilon,h}-u_{\epsilon} \Vert_{\Omega} \leq C (\epsilon^{\beta}+h^\gamma),
\end{equation}
to obtain the $\epsilon-$uniform \textcolor{black}{estimate}
$$\Vert u_{\epsilon,h}-u_{\epsilon} \Vert_{\Omega} \leq C h^{\min(\frac{\beta}{\alpha+\beta},\gamma)}. $$ 
For isotropic singular perturbations, which model diffusion phenomena in isotropic medium, many authors studied conforming and non conforming Galerkin methods for the following problem
$$ -\epsilon^{2} \Delta u_{\epsilon} +a u_{\epsilon} =f \text{, with } u=0 \text{ on } \partial \Omega,$$
where $\Omega$ \textcolor{black}{is a square} or a cube and $f$ is sufficiently regular (of class $C^2$) with some compatibility condition, that is when $f$ is zero on the edge of the vertices of $\Omega$. For instance, in \cite{Oriordan} a uniform \textcolor{black}{estimate} of the form $\vert\vert\vert u_{\epsilon,h}-u_{\epsilon} \vert\vert\vert _{\Omega}=O(h^{1/2})$ is proved, where $\vert\vert\vert \cdot \vert\vert\vert_{\Omega} $ is a variant of the energy norm. Some other quasi-uniform logarithmic \textcolor{black}{estimates} have been proved, see for instance the references \cite{Li1}, \cite{Li2}, \cite{Lin} and those cited therein.
 
In this article, we deal with anisotropic singular perturbations, which model diffusion phenomena in anisotropic medium. We
will prove some uniform \textcolor{black}{estimates} for the conforming Galerkin method in 2D and 3D for an elliptic problem with a general diffusion matrix, and a source term with low regularity. A prototype of such problems is given by the following bi-dimensional equation
\begin{equation} \label{prototype2d}
-\epsilon^{2} \partial_{x_1}^{2} u_{\epsilon}-\partial_{x_2}^{2} u_{\epsilon} =f.
\end{equation}
For the several results proved in this paper, we will consider different assumptions for the regularity of $f$, we give here the principal ones
\begin{equation} \label{fL2}
f \in L^{2}(\Omega), 
\end{equation}
\begin{equation}\label{fH2}
f \in H^{2}(\Omega).
\end{equation}
We will deal with the following general set-up of (\ref{prototype2d}) \cite{Chipot}
\begin{equation}
-\text{div}(A_{\epsilon }\nabla u_{\epsilon })=f~\text{%
in}~\Omega ,  \label{u-perturbed}
\end{equation}%
supplemented with the boundary condition 
\begin{equation}
u_{\epsilon }=0~\text{in}~\partial \Omega .  \label{u-perturbed-bord}
\end{equation}%
Here, $\Omega =\omega _{1}\times \omega _{2}$ where $\omega _{1}$ and $%
\omega _{2}$ are two bounded open sets of $%
\mathbb{R}
^{q}$ and $%
\mathbb{R}
^{N-q},$ with $N>q\geq 1$. We denote by $%
x=(x_{1},...,x_{N})=(X_{1},X_{2})\in 
\mathbb{R}
^{q}\times 
\mathbb{R}
^{N-q}$ i.e. we split the coordinates into two parts. With this notation we
set%
\begin{equation*}
\nabla =(\partial _{x_{1}},...,\partial _{x_{N}})^{T}=\binom{\nabla _{X_{1}}%
}{\nabla _{X_{2}}},\text{ }
\end{equation*}%
where 
\begin{equation*}
\nabla _{X_{1}}=(\partial _{x_{1}},...,\partial _{x_{q}})^{T}\text{ and }%
\nabla _{X_{2}}=(\partial _{x_{q+1}},...,\partial _{x_{N}})^{T}
\end{equation*}%
The matrix-valued function $A=(a_{ij})_{1\leq i,j\leq N}:\Omega \rightarrow \mathcal{M}%
_{N}(\mathbb{R})$ satisfies the classical ellipticity assumptions
\begin{itemize}
\item There exists $\lambda > 0$ such that for a.e. $x \in \Omega$ 
\begin{equation}
A(x)\xi \cdot \xi \geq \lambda \left\vert \xi \right\vert ^{2}~\text{for any}~
\xi \in 
\mathbb{R}
^{N}.  \label{hypA1}
\end{equation}

\item The elements of $A$ are bounded that is 
\begin{equation}  \label{hypA2}
a_{ij}\in L^{\infty }(\Omega )~\text{for any}~ (i,j) \in \{1,2,....,N\}^2.
\end{equation}
\end{itemize}
For the $H^2$ \textcolor{black}{estimates}, we suppose that $A$ satisfies the regularity assumption
\begin{equation}  \label{A-class-C1}
a_{ij}\in W^{1,\infty}(\bar{\Omega} )~\text{for any}~ (i,j) \in \{1,2,....,N\}^2,
\end{equation}
and the boundary condition 
\begin{equation}\label{Anullbord}
\text{ For every } i\neq j: a_{ij}=0 \text{ on } \partial \Omega.
\end{equation}
We have decomposed $A$ into four blocks
\begin{equation*}
A=%
\begin{pmatrix}
A_{11} & A_{12} \\ 
A_{21} & A_{22}%
\end{pmatrix}
\end{equation*}
where $A_{11}$, $A_{22}$ are $q\times q$ and $(N-q)\times (N-q)$ matrices
respectively. Notice that (\ref{hypA1}) implies that $A_{22}$ satisfies the ellipticity assumption
\begin{equation}
\text{ For a.e. } x \in \Omega: A_{22}(x)\xi_2 \cdot \xi_2 \geq \lambda \left\vert \xi_2 \right\vert ^{2}~\text{for any}~
\xi_2 \in 
\mathbb{R}
^{N-q}.  \label{hypA_221}
\end{equation}
For $\epsilon \in (0,1]$, we have set 
\begin{equation*}
A_{\epsilon }= 
\begin{pmatrix}
\epsilon ^{2}A_{11} & \epsilon A_{12} \\ 
\epsilon A_{21} & A_{22}%
\end{pmatrix}%
\end{equation*}%
For the \textcolor{black}{estimation} of the global rate of convergence for the continuous problem, we suppose the following additional assumption 
\begin{equation}
\text{ The block }A_{22}\text{ depends only on }X_{2}  \label{A22}
\end{equation}
The weak formulation of the problem ($\ref{u-perturbed}$)-($\ref{u-perturbed-bord}$%
) is 
\begin{equation}
\left\{ 
\begin{array}{l}
\int_{\Omega }A_{\epsilon
}\nabla u_{\epsilon }\cdot \nabla \varphi \d x=\int_{\Omega }f\text{ }\varphi
\d x, \quad \forall \varphi \in H_{0}^{1}(\Omega )\text{\ \ \ \  }
\\ 
u_{\epsilon }\in H_{0}^{1}(\Omega ),\text{\ \ \ \ \ \ \ \ \ \ \ \ \ \ \ \ \ \
\ \ \ \ \ \ \ \ \ \ \ \ \ \ \ \ \ \ \ \ \ \ \ \ \ \ \ \ \ \ \ \ \ \ \ \ \ \
\ \ \ \ \ \ \ \ \ \ \ \ \ \ \ \ \ \ \ \ \ \ \ \ \ \ \ }%
\end{array}%
\right.   \label{u-perturbed-weak}
\end{equation}%
where the existence and uniqueness is a consequence of the assumptions $(\ref%
{hypA1})$, $(\ref{hypA2})$.
The limit problem is given by
\begin{equation}
\textcolor{black}{-\text{div}_{X_{2}}(A_{22}\nabla_{X_2} u)=f~\text{in}~\Omega ,}
\label{u-limit}
\end{equation}
supplemented with the boundary condition 
\begin{equation}
u(X_{1},\cdot )=0~\text{in}~\partial \omega _{2},~\text{for}~X_{1}\in \omega
_{1}.  \label{u-limit-bord}
\end{equation}
We recall the Hilbert space \cite{ogabi4}
$$H_{0}^{1}(\Omega ;\omega _{2})=\left\{\begin{array}{cc} v\in L^{2}(\Omega )~\text{such that}%
~\nabla _{X_{2}}v\in L^{2}(\Omega )^{N-q} \\
\text{ and for a.e. } ~X_{1}\in
\omega _{1},~v(X_{1},\cdot )\in H_{0}^{1}(\omega _{2})
\end{array}\right\}, $$ 
equipped with the norm $\left\Vert \nabla _{X_{2}}(\cdot )\right\Vert
_{L^{2}(\Omega )^{N-q}}$. Notice that this norm is equivalent to 
\begin{equation*}
\left( \left\Vert (\cdot )\right\Vert _{L^{2}(\Omega )}^{2}+\left\Vert
\nabla _{X_{2}}(\cdot )\right\Vert _{L^{2}(\Omega )^{N-q}}^{2}\right) ^{1/2},
\end{equation*}%
thanks to Poincar\'{e}'s inequality 
\begin{equation}
\left\Vert v\right\Vert _{L^{2}(\Omega )}\leq C_{\omega _{2}}\left\Vert
\nabla _{X_{2}}v\right\Vert _{L^{2}(\Omega )^{N-q}},~\text{for any}~v\in
H_{0}^{1}(\Omega ;\omega _{2}).  \label{sob}
\end{equation}%
The space $H_{0}^{1}(\Omega )$ will be normed by $\left\Vert \nabla (\cdot
)\right\Vert _{L^{2}(\Omega )^{N}}$. One can check immediately that the
embedding $H_{0}^{1}(\Omega )\hookrightarrow H_{0}^{1}(\Omega ,\omega _{2})$
is continuous. The weak formulation of the limit problem (\ref{u-limit})-(\ref{u-limit-bord}) is given by 
\begin{equation}
\left\{ 
\begin{array}{ll}
\begin{array}{c}
\int_{\omega
_{2}}A_{22}\mathbf{(}X_{1},\cdot )\nabla _{X_{2}}u\mathbf{(}X_{1},\cdot
)\cdot \nabla _{X_{2}}\psi dX_{2} 
=\int_{\omega _{2}}f\mathbf{(}X_{1},\cdot )\text{ }\psi dX_{2}\text{, }%
\forall \psi \in H_{0}^{1}(\omega _{2})\text{\ \ \ \ }%
\end{array}
&  \\ 
u\mathbf{(}X_{1},\cdot )\in H_{0}^{1}(\omega _{2})\text{,\ for a.e. }%
X_{1}\in \omega _{1}\text{, \ \ \ \ \ \ \ \ \ \ \ \ \ \ \ \ \ \ \ \ \ \ \ \ \
\ \ \ \ \ \ \ \ \ \ \ \ \ \ \ \ \ \ \ \ \ \ \ \ \ \ \ \ \ \ \ \ \ \ \ \ } & 
\text{ }%
\end{array}%
\right.   
\label{u-limit-weak-bis}
\end{equation}%
where the existence and uniqueness is a consequence of the assumptions (\ref{hypA2}),(\ref{hypA_221}).
Recall that we have $u\in H_{0}^{1}(\Omega ;\omega _{2})$ and $ u_{\epsilon }\rightarrow u\text{ in }H_{0}^{1}(\Omega ;\omega _{2})$ as $\epsilon \rightarrow 0$ \cite{ogabi4}.
Notice that for $\varphi \in H_{0}^{1}(\Omega ;\omega _{2})$, then for a.e $%
X_{1}$ in $\omega _{1}$ we have $\varphi (X_{1},\cdot )\in H_{0}^{1}(\omega
_{2})$, testing with it in (\ref{u-limit-weak-bis}) and integrating over $%
\omega _{1}$ yields 
\begin{equation}
\int_{\Omega }A_{22}\nabla _{X_{2}}u\cdot
\nabla _{X_{2}}\varphi dx=\int_{\Omega }f\text{ }\varphi dx,\text{\ }\forall
\varphi \in H_{0}^{1}(\Omega ;\omega _{2}).  \label{u-limit-weak}
\end{equation}
Notice that (\ref{u-limit-weak}) could be seen as the weak formulation of the limit problem (\ref{u-limit})-(\ref{u-limit-bord}) in the Hilbert space $H_{0}^{1}(\Omega ;\omega _{2})$ (thanks to Poincaré's inequality (\ref{sob})).\\
Finally, for the estimation of the global rate of convergence, let us recall this result proved in our paper \cite{ogabi4} \textcolor{black}{(Theorem 2.3).}
\begin{equation*}
\Vert \nabla_{X_{2}}(u_{\epsilon}-u) \Vert_{L^{2}(\Omega)^{N-q}} \leq C \epsilon, \text{ for } f\in H^1_{0}(\Omega;\omega_1),
\end{equation*}
\textcolor{black}{where $C$ is independent of $\varepsilon$.} Here, $H^1_{0}(\Omega;\omega_1)$ is similarly defined as $H^1_0(\Omega;\omega_2)$. 
The main results of the article are given in two sections:
 \begin{itemize}
\item In the first section, we show a version of the previous \textcolor{black}{estimate} for more general $f$, and 
we show some high order regularity \textcolor{black}{estimates} for the solutions $u_{\epsilon}$ and $u$, for general domains in arbitrary dimension $N \geq 2$.
\item In the second section, we use the results of the first section  to analyse a $Q_{1}$ finite element method for problem (\ref{u-perturbed-weak}) when $\Omega$,  \textcolor{black}{is a square} or a cube . In the case of regular perturbation (i.e. no boundary layers), that is when $f \in H^1_{0} \cap H^2(\Omega)$, we derive a uniform \textcolor{black}{estimate} of the form $$\Vert \nabla_{X_{2}}(u_{\epsilon,h}-u_{\epsilon}) \Vert_{L^{2}(\Omega)^{N-q}} \leq C h^{\frac{1}{3}}.$$
In the case of singular perturbation (i.e. with boundary layer formation), that is when $f \in  H^2(\Omega)$, we derive a uniform \textcolor{black}{estimate} of the form $$\Vert \nabla_{X_{2}}(u_{\epsilon,h}-u_{\epsilon}) \Vert_{L^{2}(\Omega)^{N-q}} \leq C h^{\frac{1}{5}}.$$
\end{itemize}
Notice that throughout this article $C_{{f, A \text{ ... etc.}}}$ denotes a generic positive constant depending only on the objects $f, A \text{ ... etc.}$.
\section{Some results for the continuous problem}\label{sec-continue}
\subsection{Rate of convergence for general data}In this subsection, we suppose that the block $A_{12}$ satisfies the assumption 
\begin{equation}
\partial _{i}a_{ij}\in L^{\infty }(\Omega ),\partial _{j}a_{ij}\in
L^{\infty }(\Omega )\text{ for }i=1,...,q\text{ and }j=q+1,...,N.
\label{hypAd1}
\end{equation}
In \cite{ogabi4} (Theorem 2.3), we have proved the following \textcolor{black}{estimate}.
\begin{thm} \label{thm-tauxconvergence}
\cite{ogabi4} Let $\Omega =\omega _{1}\times \omega _{2}$ where $\omega
_{1}$ and $\omega _{2}$ are two bounded open sets of $%
\mathbb{R}
^{q}$ and $%
\mathbb{R}
^{N-q}$ respectively, with $N>q\geq 1.$ Suppose that $%
A $ satisfies $(\ref{hypA1})$, $(\ref{hypA2})$, $(\ref{A22})$ and $(\ref%
{hypAd1})$. Let $f\in H^1_0(\Omega,\omega_1 )$, then there exists $C_{\l,\Omega,A}>0$ such that: 
\begin{equation}
\label{tauxconvergence}
\forall \epsilon \in (0,1]:\left\Vert \nabla _{X_{2}}(u_{\epsilon
}-u)\right\Vert _{L^{2}(\Omega )^{N-q}}\leq C_{\l,\Omega,A}(\left\Vert \nabla _{X_{1}}f\right\Vert
_{L^{2}(\Omega )^{q}}+\left\Vert f\right\Vert _{L^{2}(\Omega )})\times \epsilon ,
\end{equation}%
where $u_{\epsilon }$ is the unique solution of $(\ref{u-perturbed-weak})$ in $%
H_{0}^{1}(\Omega )$ and $u$ is the unique solution to $(\ref{u-limit-weak})$
in $H_{0}^{1}(\Omega ;\omega _{2}),$  moreover we have $u \in H^1_0(\Omega).$
\end{thm}
When $f$ does not have the $H^1_0$ regularity in the $X_1$ direction i.e. $f \notin H^1_0(\Omega,\omega_1 )$, we will show a rate of convergence of order $O(\epsilon^{s})$, $0<s<1$. The argument is based on the interpolation trick (see for instance \cite{Lions}). In the above reference, Lions shows that every $H^1$ function could be \textcolor{black}{split} into a $H^1_0$-function with a big $H^1$ norm, and a $H^1$-function with a small $L^2$ norm, when the domain $\Omega$ is regular (of class $C^1$, or of class $C^2$ for the decomposition of $H^2$ functions). Here, we will prove some decomposition lemmas for functions with a partial $H^1$ regularity and for domains with a very low regularity.
\begin{deff}\label{def-decomp}
Let $\mathcal{O}$ be a bounded open set of $\mathbb{R}^d$. we say that $\mathcal{O}$ satisfies the decomposition hypothesis if there exist positive constants $c_1$ and $c_2$ such that for any $\delta \in (0,1)$ there exists $\mathcal{O}_{\delta} \subset \subset \mathcal{O}$ open such that 
\begin{equation} \label{D}
\text{\textit{meas}}( \mathcal{O} \textbackslash \mathcal{O}_{\delta}) \leq c_1 \delta  \text{  and } 
\text{dist}(\mathcal{O}_{\delta}, \partial \mathcal{O}) \geq c_2 \delta \tag{D}
\end{equation}
\end{deff}
One can show that polygonal domains, or more generally Lipschitz domains, satisfy the decomposition hypothesis (\ref{D}). That is still true for some non-Lipschitz domains, for example, for the following open set of $\mathbb{R}^2$,  $\mathcal{O}=\{\vert y \vert < x^2, \text{ } 0<x<1 \}$. Notice that there exist bounded open sets which do not satisfy the decomposition propriety, for instance in dimension $d \geq 1$, the open set $\mathbb{B}_d(0,1) \backslash \{\cup_{k=1}^{\infty}\mathbb{S}_{d}(0,\frac{1}{k})\}$, where $\mathbb{B}_d(0,1)$ is the \textcolor{black}{unit} euclidean ball of $\mathbb{R}^d$ and  $\mathbb{S}_{d}(0,\frac{1}{k})$ is the euclidean sphere of $\mathbb{R}^d$ of center $0$ and ray $\frac{1}{k}$,  gives such an example. Now, we have to prove the following

\begin{lem}\label{lem-decomp}
Let $\Omega=\omega_1 \times \omega_2 $ be a bounded open set of $\mathbb{R}^q \times \mathbb{R}^{N-q}$. Let $f \in L^r(\Omega)$ for some $r>2$ such that $\nabla_{X_1} f \in L^2(\Omega)^q$. Suppose that $\omega_1$ satisfies (\ref{D}), then there exist $C_{_{f,c_1,c_2,\omega_2}},C_{_{f,c_1,\omega_2}}>0$ such that for any $\delta \in (0,1)$ there exist $f^1_{\delta} \in H^1_0(\Omega, \omega_1)$ and $f^2_{\delta} \in L^2(\Omega)$ with  $\nabla_{X_1} f^2_{\delta} \in L^2(\Omega)^q$ such that $ f=f^1_{\delta}+f^2_{\delta}$ with:  $$ \Vert \nabla_{X_1} f^1_{\delta} \Vert _{L^2(\Omega)^q} \leq C_{_{f,c_1,c_2,\omega_2}} \delta^{-\frac{1}{2}-\frac{1}{r}} \text{, and } \Vert f^2_{\delta} \Vert _{L^2(\Omega)} \leq C_{_{f,c_1,\omega_2}} \delta^{\frac{1}{2}-\frac{1}{r}}.$$
\textcolor{black}{Here, the positive numbers $c_1$, $c_2$ are the constants of the decomposition of $\Omega$ given in Definition \ref{def-decomp}.}
\end{lem}
\begin{proof}
Let us recall the space
$$
H_{0}^{1}(\Omega ;\omega _{1})=\left\{\begin{array}{cc}

v\in L^{2}(\Omega )~\text{such that}%
~\nabla _{X_{1}}v\in L^{2}(\Omega )^{q} \\ \text{ and for a.e. }~X_{2}\in
\omega _{2},v(\cdot, X_{2} )\in H_{0}^{1}(\omega _{1})\end{array} \right\},
$$
and the anisotropic Poincaré's inequality 
\begin{equation}
\left\Vert v\right\Vert _{L^{2}(\Omega )}\leq C_{\omega _{1}}\left\Vert
\nabla _{X_{1}}v\right\Vert _{L^{2}(\Omega )^{q}},~\text{for any}~v\in
H_{0}^{1}(\Omega ;\omega _{1}).  \label{sob-bis}
\end{equation}%
Let $\delta \in (0,1)$, since $\omega_1$ satisfies (\ref{D}) then  there exists $\omega_1^{\delta} \subset \subset \omega_1$ such that $$ \text{\textit{meas}}( \omega_1 \textbackslash \omega_1^{\delta}) \leq c_1 \delta  \text{  and } 
\text{dist}(\omega_1^{\delta}, \partial \omega_1) \geq c_2 \delta $$
Let $K_{\delta} = \{ X_1 \in \omega_1 \text{, dist}(X_1,\omega_1^{\delta}) \leq \frac{c_2\delta}{3}\}  $. We define the bump function $\rho_{\delta} \in \mathcal{D}(\omega_1)$ by 
$$\rho_{\delta}(X_1)=\int_{K_{\delta}} \theta_{\frac{c_2\delta}{3}}(X_1-y)dy, $$
where $\theta_{\mu}(\cdot)=\Big(\frac{1}{\mu}\Big)^{q}\theta(\frac{\cdot}{\mu}) $, with $\theta \in \mathcal{D}(\mathbb{R}^q)$ such that $\int_{\mathbb{R}^q} \theta(X_1)dX_1=1$ and $\text{Supp}(\theta) \subset B_{q}(0,1)$. We can check that $$\rho_{\delta}=1 \text{ on }\omega_1^{\delta}, \text{ supp}(\rho_{\delta}) \subset \omega_1 \text{, and }\Vert \rho_{\delta} \Vert_{\infty} \leq 1  \text{, } \Vert \nabla_{X_1}\rho_{\delta} \Vert_{\infty} \leq \frac{3\Vert \nabla_{X_1}\theta \Vert_{\infty}}{c_2\delta}.$$
We define $$ f_{\delta}^1:= \rho_{\delta}f \text{ and } f_{\delta}^2:= f-\rho_{\delta}f.$$
It is clear that $f_{\delta}^1 \in H^1_0(\Omega;\omega_1)$ and $f_{\delta}^2 \in L^2(\Omega) $ with $\nabla_{X_1}f_{\delta}^2 \in L^2(\Omega)^q $. \\
Now, let $1\leq i \leq q$, we have 
\textcolor{black}{$$\Vert \partial_i f^1_{\delta} \Vert_{L^2(\Omega)} \leq \Vert \rho_{\delta}\partial_i f \Vert_{L^2(\Omega)}+\Vert f \partial_i \rho_{\delta} \Vert_{L^2((\omega_1\textbackslash \omega_1^{\delta}) \times \omega_2)}, $$}
therefore
$$\Vert \partial_i f^1_{\delta} \Vert_{L^2(\Omega)} \leq \Vert \partial_i f \Vert_{L^2(\Omega)}+\frac{3\Vert \partial_{i}\theta \Vert_{\infty}}{c_2\delta} \Vert f \Vert_{L^2((\omega_1\textbackslash \omega_1^{\delta}) \times \omega_2)},$$
whence by H\"older inequality we get 
$$\Vert \partial_i f^1_{\delta} \Vert_{L^2(\Omega)} \leq \Vert \partial_i f \Vert_{L^2(\Omega)}+\frac{3\Vert \partial_{i}\theta \Vert_{\infty}}{c_2\delta} \Vert f \Vert_{L^r(\Omega)} \times \Big(c_1\vert \omega_2 \vert \delta \Big)^{\frac{1}{2}-\frac{1}{r}}, $$
and the first inequality of Lemma \ref{lem-decomp} follows. Similarly, we obtain the second inequality of Lemma \ref{lem-decomp}.
\end{proof}
Now, we are able to prove the following theorem
\begin{thm}\label{thm-tauxconvergence-H1}
Assume that $\Omega$ and $A$ are given as in Theorem \ref{thm-tauxconvergence}, assume in addition that $\omega_1$ satisfies (\ref{D}). Suppose that $f \in L^r(\Omega)$ for some $\infty>r>2$ such that $\nabla_{X_1} f\in L^2(\Omega)$, then there exists $C_{\l,f,A,\Omega}>0$ such that 
$$ \forall \epsilon \in (0,1],~\left\Vert \nabla _{X_{2}}(u_{\epsilon
}-u)\right\Vert _{L^{2}(\Omega )^{N-q}}\leq C_{_{\l,f,A,\Omega}} \times \epsilon^{\frac{1}{2}-\frac{1}{r}}.$$
In particular, when $r=\infty$ we have $$ \forall \epsilon \in (0,1],~\left\Vert \nabla _{X_{2}}(u_{\epsilon
}-u)\right\Vert _{L^{2}(\Omega )^{N-q}}\leq C_{_{\l,f,A,\Omega}} \times \epsilon^{\frac{1}{2}}.$$
\end{thm}
\begin{proof}
Let $f=f^1_{\delta}+f^2_{\delta} $ be the decomposition of Lemma \ref{lem-decomp}. Let $u_{\epsilon}^1$ (resp. $u_{\epsilon}^2$) be the solution of (\ref{u-perturbed-weak}) with $f$ replaced by $f_{\delta}^1$ ( resp. $f_{\delta}^2$). The linearity of the problem shows that 
\begin{equation}\label{u-decomp}
u_{\epsilon}=u_{\epsilon}^1+u_{\epsilon}^2.
\end{equation}
Similarly, Let $u^1$ (resp. $u^2$) be the solution of (\ref{u-limit-weak}) with $f$ replaced by $f_{\delta}^1$ ( resp. $f_{\delta}^2$), we also have 
\begin{equation}\label{u-limit-decomp}
u=u^1+u^2.
\end{equation}
Now, according to Theorem \ref{thm-tauxconvergence} we have 
\begin{equation}\label{f^1}
\left\Vert \nabla _{X_{2}}(u_{\epsilon
}^1-u^1)\right\Vert _{L^{2}(\Omega )^{N-q}}\leq C_{_{\l,\Omega,A}}(\left\Vert \nabla _{X_{1}}f_{\delta}^1\right\Vert
_{L^{2}(\Omega )^{q}}+\left\Vert f_{\delta}^1\right\Vert _{L^{2}(\Omega )})\times \epsilon ,
\end{equation}
By using the anisotropic Poincaré's inequality (\ref{sob-bis}), which holds because $f^1_{\delta} \in H^1_0(\Omega, \omega_1)$), we obtain from (\ref{f^1}) 
\begin{equation}
\left\Vert \nabla _{X_{2}}(u_{\epsilon
}^1-u^1)\right\Vert _{L^{2}(\Omega )^{N-q}}\leq C_{_{\l,\Omega,A}} \times \left\Vert \nabla _{X_{1}}f_{\delta}^1\right\Vert
_{L^{2}(\Omega )^{q}}\times \epsilon.
\end{equation}
By applying Lemma \ref{lem-decomp} on the right hand side of the previous inequality we get 
\begin{equation}\label{f^1-bis}
\left\Vert \nabla _{X_{2}}(u_{\epsilon
}^1-u^1)\right\Vert _{L^{2}(\Omega )^{N-q}}\leq C_{_{\l,f,\Omega,A}} \times \delta^{-\frac{1}{2}-\frac{1}{r}}\times \epsilon.
\end{equation}
Now, let us \textcolor{black}{estimate} $\left\Vert \nabla _{X_{2}}(u_{\epsilon
}^2-u^2)\right\Vert _{L^{2}(\Omega )^{N-q}}$. We test by $u_{\epsilon
}^2$ (resp. $u^2$) in the corresponding weak formulation i.e. (\ref{u-perturbed-weak}) with $f$ replaced by $f_{\delta}^2$ ( resp. (\ref{u-limit-weak}) with $f$ replaced by $f_{\delta}^2$), we get 
$$ \left\Vert \nabla _{X_{2}}u_{\epsilon
}^2\right\Vert _{L^{2}(\Omega )^{N-q}} \leq C_{_{\Omega,\l}} \Vert f_{\delta}^2 \Vert_{L^2(\Omega)} \text{ and } \left\Vert \nabla _{X_{2}}u^2\right\Vert _{L^{2}(\Omega )^{N-q}} \leq C_{_{\Omega,\l}} \Vert f_{\delta}^2 \Vert_{L^2(\Omega)}.
$$
Therefore, by the triangle inequality and Lemma \ref{lem-decomp} we get 
\begin{equation}\label{f^2}
\left\Vert \nabla _{X_{2}}(u_{\epsilon
}^2-u^2)\right\Vert _{L^{2}(\Omega )^{N-q}}\leq C_{_{f,\Omega,\l}} \times \delta^{\frac{1}{2}-\frac{1}{r}}
\end{equation}
Finally, by using the decompositions (\ref{u-decomp}), (\ref{u-limit-decomp}), and inequalities (\ref{f^1-bis}), (\ref{f^2}) and the triangle inequality we get 
\begin{equation*}
\left\Vert \nabla _{X_{2}}(u_{\epsilon
}-u)\right\Vert _{L^{2}(\Omega )^{N-q}}\leq C_{_{f,\Omega,A,\l}} \Big( \delta^{-\frac{1}{2}-\frac{1}{r}} \epsilon+\delta^{\frac{1}{2}-\frac{1}{r}}\Big)
\end{equation*}
Notice that this last inequality holds for every $\delta \in (0,1)$ and every $\epsilon \in (0,1]$. Whence, we can choose $\delta=\epsilon$ and we get the expected \textcolor{black}{estimate}. The case $r=\infty$ follows immediately by letting $r \rightarrow \infty$ in the first \textcolor{black}{estimate}.
\end{proof}
Let us finish by this particular case which will be used in the analysis of the numerical scheme. In fact, we have the following analogous of Lemma \ref{lem-decomp} 
\begin{lem}\label{lem-decomp-bis}
Let $\Omega$ be a bounded open set of $\mathbb{R}^N$. Let $f \in H^2(\Omega) \cap L^{\infty}(\Omega)$. Suppose that $\Omega$ satisfies $(\ref{D})$, then the exist $C_{_{f,c_1,\Omega}},C_{_{f,c_1,c_2,\Omega}}>0$ such that for any $\delta \in (0,1)$ there exist $f^1_{\delta} \in H^1_0 \cap H^2(\Omega)$ and $f^2_{\delta} \in H^2(\Omega)$ such that $ f=f^1_{\delta}+f^2_{\delta} $ with: $$ \Vert f^1_{\delta} \Vert _{H^1(\Omega)} \leq C_{_{f,c_1,c_2,\Omega}} \delta^{-1/2},\text{  } \Vert f_{\delta}^1\Vert_{L^2(\Omega)}\leq \Vert f\Vert_{L^2(\Omega)},$$ and $$\Vert f_{\delta}^1 \Vert _{H^2(\Omega)} \leq C_{_{f,c_1,c_2,\Omega}} \delta^{-\frac{3}{2}} \text{ and } \Vert f^2_{\delta} \Vert _{L^2(\Omega)} \leq C_{_{f,c_1,\Omega}} \delta^{1/2}.$$
\end{lem}
The proof is similar to that of Lemma \ref{lem-decomp}. Here, we can notice that if $\Omega$ is Lipschitz then the assumption $f \in L^{\infty}(\Omega)$, is automatically satisfied in dimension 2 and 3 thanks to Sobolev embeddings.

\subsection{Some Regularity \textcolor{black}{estimates} for the solution of the perturbed problem}
In this subsection, we suppose that $\Omega$ is of the form 
\begin{equation} \label{Omega-cuboid}
\Omega=\prod\limits_{i=1}^{N}(0,l_{i}),
\end{equation}
where $l_{1},...,l_{N}$ are positive real numbers, and $N \geq 2$. We will prove the following 
\begin{thm} \label{theorem-H2-upertubed}
Assume $(\ref{hypA1})$, $(\ref{A-class-C1})$, and $(\ref{Anullbord})$. Let $\Omega$ such that $(\ref{Omega-cuboid})$. Let $f$ such that $(\ref{fL2})$ is satisfied. Let $u_{\epsilon}$ be the solution of $(\ref{u-perturbed-weak})$ then there exists a constant $C_{_{\l,f,A,\Omega}}>0$ such that, for any $\epsilon \in (0,1]$
\begin{equation*}
\epsilon^{2}\Vert D^{2}_{X_1}u_{\epsilon} \Vert _{L^{2}(\Omega)^{q^2}}+ \epsilon\Vert D^{2}_{X_1X_2}u_{\epsilon} \Vert _{L^{2}(\Omega)^{q(N-q)}} + \Vert D^{2}_{X_2}u_{\epsilon} \Vert _{L^{2}(\Omega)^{(N-q)^2}} \leq C_{_{\l,f,A,\Omega}}
\end{equation*}
In addition, if $A \in C^1(\bar{\Omega})$ then the strong convergences
$$ \epsilon^2 D^{2}_{X_1}u_{\epsilon} \rightarrow 0, \text{ }D^{2}_{X_1X_2}u_{\epsilon} \rightarrow 0 \text{ and }D^{2}_{X_2}u_{\epsilon} \rightarrow D^{2}_{X_2}u,  $$
hold in $L^{2}(\Omega)^{q^2}$, $L^{2}(\Omega)^{q(N-q)}$ and $L^{2}(\Omega)^{(N-q)^2}$ respectively. 

\end{thm}
Here, we used the notation: $D^{2}_{X_1}:=(\partial_{ij}^2)_{1\leq i,j \leq q}$, $D^{2}_{X_2}:=(\partial_{ij}^2)_{1+q\leq i,j \leq N}$, and $D^{2}_{X_1X_2}:=(\partial_{ij}^2)_{1\leq i\leq q,\text{ } q+1\leq j \leq N}$\\ 
In \cite{ogabi2}, we have a local version of this result.
The proof of Theorem \ref{theorem-H2-upertubed} is based on a symmetrization trick and the application of the local result.

We introduce $\widetilde{\Omega}=\prod\limits_{i=1}^{N}(-l_{i},2l_{i})$, and we denote $sub(\widetilde{\Omega})$ the set of all disjoint subsets of $\widetilde{%
\Omega}$ of the form $I=\prod\limits_{i=1}^{N}I_{i}$ where each $I_{i}$ has
one of the forms $(-l_{i},0),$ $(0,l_{i}),$ $(l_{i},2l_{i}).$ For $0 \leq j \leq N$ fixed
we denote $sub(\widetilde{\Omega})_{-1,j}$, $sub(\widetilde{\Omega})_{1,j},$ $sub(%
\tilde{\Omega})_{2,j}$ the subsets of $sub(\widetilde{\Omega})$ such that $I_{j}$
is of the form $(-l_{j},0),$ $(0,l_{j}),$ $(l_{j},2l_{j})$ respectively. It
is clear that these three subsets define a partition of $sub(\widetilde{\Omega}).
$ For $x=(x_{i})_{1\leq i\leq N}\in I$, we denote 
\begin{equation*}
y=(s_{i}(I)+r_{i}(I)x_{i})_{1\leq i\leq N}\in \Omega 
\end{equation*}%
where $r_{i}$ and $s_{i}$, ${1\leq i\leq N}$, are defined on $sub(\widetilde{\Omega})$ by : 
\begin{equation*}
r_{i}(I)=\left\{
\begin{array}{c}
{1\text{ if }I\in sub(\widetilde{\Omega})_{1,i}} \\ 
{-1\text{
else\ \ \ \ \ \ \ \ \ \ \ }}
\end{array}
\right.
\end{equation*}%
\begin{equation*}
s_{i}(I)=\left\{ 
\begin{array}{c}
0\text{ if }I\in sub(\widetilde{\Omega})_{-1,i}\cup sub(\widetilde{\Omega})_{1,i} \\ 
2l_{i}\text{ if }I\in sub(\widetilde{\Omega})_{2,i}\text{ \ \ \ \ \ \ \ \ \ \ \
\ \ \ \ \ }%
\end{array}%
\right. 
\end{equation*}
Now, we define the function $\widetilde{f}\in L^{2}(\widetilde{\Omega})$ by 
$$
\text{ For }x \in I\in sub(\widetilde{\Omega}):\widetilde{f}(x)=\left(
\prod\limits_{i=1}^{N}r_{i}(I)\right) \times f(y)_{1\leq i\leq N}), \\
\text{ and }\widetilde{f} =0\text{ else}.
$$
Similarly, we define the extension $\widetilde{u}_{\epsilon }$ of $u_{\epsilon }$ by:%
$$
\text{ For }x \in I\in sub(\widetilde{\Omega}):\widetilde{u}_{\epsilon }(x)=\left(
\prod\limits_{i=1}^{N}r_{i}(I)\right) \times u_{\epsilon }(y), \\
\text{ and }\widetilde{u}_{\epsilon } =0\text{ else}.
$$
We define $\widetilde{A}=(\widetilde{a_{ij}})$ the extension of $A$ as follows: \
For $x\in \overline{\widetilde{\Omega}}$, there exists $I\in sub(\widetilde{\Omega})$
such that $x\in \overline{I},$ in this case we set 
\begin{equation*}
\widetilde{a_{ij}}(x)=r_{i}(I)r_{j}(I)a_{ij}(y), \text{ } i,j=1,...,N.
\end{equation*}%
Notice that assumption (\ref{Anullbord}) implies that the value of each $\widetilde{A}(x)$ does not \textcolor{black}{ depend on the choice} of $I$ so $\widetilde{A}$ is well-defined. Notice also that (\ref{A-class-C1}) implies that $\widetilde{A}$ is Lipschitz on $\overline{\widetilde{\Omega}}$. Moreover, we can check immediately that $\widetilde{A}$ satisfies the ellipticity assumption (\ref{hypA1}) with the same constant. 
Finally, we define $\widetilde{A}_{\epsilon}$ as we have defined $A_{\epsilon}$ (see the introduction). \\ Under the above notations we have the following Lemma:
\begin{lem} \label{lem-regularity-prolonge}
Suppose that assumptions of Theorem \ref{theorem-H2-upertubed} hold. Let $\widetilde{u}_{\epsilon }$ \textcolor{black}{be} constructed as above, then $\widetilde{u}_{\epsilon }$ is the unique
weak solution in $H_{0}^{1}(\widetilde{\Omega })$ to the elliptic equation 
\begin{equation*}
-\text{div}(\widetilde{A}_{\epsilon }\nabla \widetilde{u}_{\epsilon })=\widetilde{f}%
\text{. }
\end{equation*}
Moreover, we have $\widetilde{u}_{\epsilon } \in H^{2}_{loc}(\widetilde{\Omega})$.
\end{lem}

\begin{proof}
At first, one can check immediately that the restriction of $\widetilde{u}_{\epsilon }$ to
each $I\in Sub(\widetilde{\Omega})$ belongs to $H_{0}^{1}(I)$ and hence $\widetilde{u%
}_{\epsilon }\in H_{0}^{1}(\widetilde{\Omega}),$ and moreover for $1\leq j\leq N$
and $I\in sub(\widetilde{\Omega}),$ we have : 
\begin{equation*}
\text{For a.e. } x\in I:\partial _{j}\widetilde{u}_{\epsilon }(x)=r_{j}(I)\left(
\prod\limits_{i=1}^{N}r_{i}(I)\right) \partial _{j}u_{\epsilon }(y).
\end{equation*}%
Now, let $\varphi \in \mathcal{D}(\widetilde{\Omega})$ we have%
\begin{equation*}
\int_{\widetilde{\Omega}}\widetilde{A}_{\epsilon }\nabla \widetilde{u}_{\epsilon }\cdot
\nabla \varphi dx=\sum_{I\in Sub(\widetilde{\Omega})}\int_{I}\widetilde{A}_{\epsilon
}\nabla \widetilde{u}_{\epsilon }\cdot \nabla \varphi dx,
\end{equation*}%
by a change of variables we get 
\begin{equation}
\int_{\widetilde{\Omega}}\widetilde{A}_{\epsilon }\nabla \widetilde{u}_{\epsilon }\cdot
\nabla \varphi dx=\sum_{I\in Sub(\widetilde{\Omega})}\int_{\Omega }A_{\epsilon
}\nabla u_{\epsilon }\cdot \nabla \widetilde{\varphi}_{I}dx,
\label{variables-change-ch2}
\end{equation}%
where $\widetilde{\varphi}_{I}$ is defined on $\overline{\Omega }$ by 
\begin{equation*}
\widetilde{\varphi}_{I}(x)=\left( \prod\limits_{i=1}^{N}r_{i}(I)\right) \varphi_{I}
((s_{i}(I)+r_{i}(I)x_{i})_{1\leq i \leq N}),
\end{equation*}%
and $\varphi _{I}$ is the restriction of $\varphi $ on $I$. Let us show that 
\begin{equation}
\sum_{I\in Sub(\widetilde{\Omega})}\widetilde{\varphi}_{I}\in H_{0}^{1}(\Omega ).
\label{phi-variable-change-ch2}
\end{equation}%
It is clear that $\sum_{I\in sub(\widetilde{\Omega})}\widetilde{\varphi}_{I}\in
H^{1}(\Omega )\cap C(\bar{\Omega}),$ it is enough to show that it vanishes
on $\partial \Omega .$ So, let $x^{0}=(x_{i}^{0})_{1\leq i\leq N}$ be an element of $%
\partial \Omega ,$ then there exists at least $1\leq j\leq N$ such that $x^{0}_{j}=0$ or $%
x^{0}_{j}=l_{j}.$

1) If $x_{j}^{0}=0:$ For any $I \in sub(\widetilde{\Omega})_{2,j}$, we have $%
y_{j}^{0}=s_{j}(I)+r_{j}(I)x_{j}^{0}=2l_{j}$, then $y^{0}\in \partial \widetilde{%
\Omega}$ therefore $\widetilde{\varphi}_{I}(x)=0.$ Now, for any $I\in sub(\widetilde{%
\Omega})_{-1,j}$, we have $y_{j}^{0}=s_{j}(I)+r_{j}(I)x_{j}^{0}=-x_{j}^{0}=0$, %
and any $I\in sub(\widetilde{\Omega})_{1,j}$, we have : $%
y_{j}^{0}=s_{j}(I)+r_{j}(I)x_{j}^{0}=x_{j}^{0}=0$, notice that there is a
bijection from $sub(\widetilde{\Omega})_{1,j}$ onto $sub(\widetilde{\Omega})_{-1,j}$
defined by : $I\mapsto I^\prime$ such that  $I$ and $%
I^\prime$ have the same intervals except for the $j^{th}$ one we have $%
I_{j}=(0,l_{j})$ and $I_{j}^\prime =(-l_{j},0).$ For such $I$ and $%
I^\prime$ we have $r_{j}(I)=1$ and $ r_{j}%
(I^\prime)=-1$, then $\widetilde{\varphi}_{I}(x^{0})+\widetilde{\varphi}_{I^\prime
}(x^{0})=0.$ Finally, we get $\sum_{I\in Sub(\widetilde{\Omega})}\widetilde{\varphi}%
_{I}(x^{0})=0.$

2) If $x_{j}^{0}=l_{j}:$ For any $I\in sub(\widetilde{\Omega})_{-1,j}$, $\
y_{j}^{0}=s_{j}(I)+r_{j}(I)x_{j}^{0}=-l_{j}$, then $y^{0}\in \partial \widetilde{%
\Omega}$ therefore $\widetilde{\varphi}_{I}(x)=0$. Now, for any $I\in sub(\widetilde{%
\Omega})_{1,j}$, we have $y_{j}^{0}=s_{j}(I)+r_{j}(I)x_{j}^{0}=l_{j}$, and any 
$I\in sub(\widetilde{\Omega})_{2,j}$, we have : $y_{j}^{0}=2l_{j}-l_{j}=l_{j}$,
notice that there is a bijection from $sub(\widetilde{\Omega})_{1,j}$ onto $sub(%
\widetilde{\Omega})_{2,j}$ defined by : $I\longmapsto I^\prime$ such
that $I$ and $I^\prime$ have the same intervals except for the $j^{th}$
one we have $I_{j}=(0,l_{j})$ and $I_{j}^\prime=(l_{j},2l_{j}).$ For such 
$I$ and $I^\prime$ we have $r_{j}(I)=1$ and $ r_{j}%
(I^\prime)=-1$, then  $\widetilde{\varphi}_{I}(x^{0})+\widetilde{\varphi}%
_{I^\prime}(x^{0})=0$. Finally, we get $\sum_{I\in Sub(\widetilde{\Omega})}\widetilde{\varphi}%
_{I}(x^{0})=0$. \\
At the end, (\ref{phi-variable-change-ch2}) follows from the two points above. \\
Now, since $u_{\epsilon }$ is the solution of $(\ref{u-perturbed-weak})
$  then (\ref{variables-change-ch2}) and (\ref%
{phi-variable-change-ch2}) give    
\begin{equation*}
\int_{\widetilde{\Omega}}\widetilde{A}_{\epsilon }\nabla \widetilde{u}_{\epsilon }\cdot
\nabla \varphi dx=\int_{\Omega }f\sum_{I\in Sub(\widetilde{\Omega})}\widetilde{%
\varphi}_{I}dx.
\end{equation*}%
By using another variables change in the second member of the above equality we get
the first affirmation of the Lemma.
Finally, as we have mentioned above, the function $\widetilde{A_{\epsilon}}$ is Lipschitz on $\widetilde{\Omega}$ (thanks to (\ref{A-class-C1})), then the $H^{2}$ interior elliptic regularity gives the second affirmation of the Lemma.
\end{proof}
Now, we can prove Theorem \ref{theorem-H2-upertubed}. 
Let $\omega \subset \subset \Omega$ an open set. According to Corollary 2.3 of \cite{ogabi2} we have, for any $\epsilon \in (0,1]$ 
\begin{equation} \label{H2borne1}
\epsilon^{2}\Vert D^{2}_{X_1}u_{\epsilon} \Vert _{L^{2}(\omega)^{q^2}}+\epsilon\Vert D^{2}_{X_1X_2}u_{\epsilon} \Vert _{L^{2}(\omega)^{q(N-q)}}+ \Vert D^{2}_{X_2}u_{\epsilon} \Vert _{L^{2}(\omega)^{(N-q)^2}} \leq C_{_{\l,f,A,\omega}}.
\end{equation}
Now, let us show the same \textcolor{black}{estimate} near the boundary of $\Omega$. Let $\omega ' \subset \subset \widetilde{\Omega}$, then by using Corollary 2.3 of \cite{ogabi2} ,thanks to Lemma \ref{lem-regularity-prolonge}, we get 
\begin{equation*} 
\epsilon^{2}\Vert D^{2}_{X_1}\widetilde{u}_{\epsilon} \Vert _{L^{2}(\omega')^{q^2}}+\epsilon\Vert D^{2}_{X_1X_2}\widetilde{u}_{\epsilon} \Vert _{L^{2}(\omega')^{q(N-q)}}+ \Vert D^{2}_{X_2}\widetilde{u}_{\epsilon} \Vert _{L^{2}(\omega')^{(N-q)^2}} \leq C_{_{\l,f,A,\omega'}}.
\end{equation*}
Therefore, 
\begin{equation} \label{H2borne2}
\epsilon^{2}\Vert D^{2}_{X_1}u_{\epsilon} \Vert _{L^{2}(\omega ' \cap \Omega)}+\epsilon\Vert D^{2}_{X_1X_2}u_{\epsilon} \Vert _{L^{2}(\omega ' \cap \Omega)}+ \Vert D^{2}_{X_2}u_{\epsilon} \Vert _{L^{2}(\omega ' \cap \Omega)} \leq C_{_{\l,f,A,\omega'}},
\end{equation}
By compacity, we can cover $\bar{\Omega}$ by a finite cover of open subsets of $\omega$-type and $\omega ' \cap \Omega$-type, then we use (\ref{H2borne1})-(\ref{H2borne2}), and the \textcolor{black}{estimates} of the Theorem \ref{theorem-H2-upertubed} follows.\\
 For the convergences of Theorem \ref{theorem-H2-upertubed}, we use the same trick, in fact when $A \in C^1(\bar{\Omega})$, then $\widetilde{A}$ satisfies eq. (13) in \cite{ogabi2}. 

\subsection{$H^{2}$ Regularity of the solution of the limit problem}
In this subsection, we consider a general bounded domain $\Omega =\omega_1 \times \omega_2$. 
\begin{thm} \label{thm-regularityH2-u} 
Let $\Omega =\omega _{1}\times \omega _{2}$ be an open bounded subset of $\R^N$,  where $\omega
_{1}$ and $\omega _{2}$ are two bounded open subsets of $
\mathbb{R}^{q}$ and $\mathbb{R}^{N-q},$ with $N>q\geq 1.$ Let us assume that $\omega_2$ is convex.
Let  $f$ such that $(\ref{fL2})$ is satisfied, and such that $
\nabla_{X_1}f\in L^{2}(\Omega )^q,$ $D_{X_1}^{2}f\in L^{2}(\Omega )^{q^2}$. Assume that  $A_{22}$ satisfies $(\ref{hypA_221}) $ and $(\ref{A22})$. Assume in addition $(\ref{A-class-C1})$ and  let $u$ be the unique solution in $H_{0}^{1}(\Omega ;\omega _{2})$ of $(%
\ref{u-limit-weak})$, then $u\in H^{2}(\Omega )$ and
$$\Vert u \Vert_{H^2(\Omega)} \leq C_{_{\l,A,\Omega}}\Big(\Vert f \Vert_{L^2(\Omega)}+\Vert \nabla_{X_1} \Vert_{L^2(\Omega)^{q}}+\Vert D^2_{X_1} f\Vert_{L^2(\Omega)^{q^2}}\Big).$$
\end{thm}
We will proceed in several steps to prove  Theorem \ref{thm-regularityH2-u}. 
In the following lemma we prove that $D^2_{X_2} u $ is a function of $L^2(\Omega)^{N-q}$.
\begin{lem}
\label{lem-regul1} 
Under assumptions of Theorem \ref{thm-regularityH2-u} we have%
\begin{equation*}
D_{X_2}^{2} u\in L^{2}(\Omega )^{(N-q)^2}~\text{ and }~ \| D_{X_2}^2 u \|_{L^2(\Omega)^{(N-q)^2}} \leq  C_{_{\l,A_{22},\omega_2}} \| f \|_{L^2(\Omega)}.
\end{equation*}
\end{lem}

\begin{proof}
We proceed in several steps.

\textbf{Step 1.} Let us assume that $f(x) = f_1(X_1) f_2(X_2)$ where $f_1 \in L^2(\omega_1)$ and $f_2 \in L^2(\omega_2)$. Let $ u_{f_2} $ be the unique solution of 
\begin{equation*} 
\left\{ 
\begin{array}{l}
\int_{\omega _{2}}A_{22}\nabla_{X_2}u_{f_2} \cdot  \nabla_{X_2}\varphi_2 d X_{2}=\int_{\omega _{2}}f_2\varphi _{2} d X_{2}\text{, \ }%
\forall \varphi _{2}\in H_{0}^{1}(\omega _{2}), \\
u_{f_2} \in H_0^1(\omega_2).
\end{array}%
\right. ,  
\end{equation*}%
According to assumption (\ref{A-class-C1}) it follows that $A_{22} \in W^{1,\infty}(\bar{\omega_{2}})$,  and since $\omega_2$ is convex we obtain by the elliptic regularity in $\omega_{2}$ that the function $u_{f_2}$ belongs to $H^2(\omega_2)$. We multiply the previous identity by $ f_1(X_1) \varphi_1(X_1) $ where $\varphi_1 \in H^1_0(\omega_1)$ and we integrate over $\omega_1$, then we use assumption (\ref{A22}) to obtain 
$$
\int_{\Omega}A_{22}\nabla_{X_2} (f_1u_{f_2}) \cdot  \nabla_{X_2} ( \varphi_1 \varphi_2) d x=\int_{\Omega}f\varphi_1\varphi _{2} d x \text{, \ }%
\forall (\varphi_1,\varphi _{2})\in H_0^1(\omega_1) \times H_{0}^{1}(\omega _{2}),
$$
which gives by linearity
$$
\int_{\Omega}A_{22}\nabla_{X_2} (f_1u_{f_2}) \cdot  \nabla_{X_2} \varphi d x=\int_{\Omega}f\varphi d x\text{, \ }%
\forall \varphi\in H_0^1(\omega_1) \otimes H_{0}^{1}(\omega _{2}).
$$
Using the fact that $H_0^1(\omega_1) \otimes H_0^1(\omega_2)$ is dense in $H_{0}^{1}(\Omega ;\omega _{2})$ \cite{ogabi4}, we obtain
\begin{equation*}
\left\{ 
\begin{array}{l}
\int_{\Omega}A_{22}\nabla_{X_2} (f_1u_{f_2}) \cdot  \nabla_{X_2} \varphi d x=\int_{\Omega}f\varphi d x\text{, \ }%
\forall \varphi \in H_{0}^{1}(\Omega ;\omega _{2}) \\
f_1 u_{f_2} \in H_{0}^{1}(\Omega ;\omega _{2}).
\end{array}%
\right. 
\end{equation*}%
Consequently we obtain that $$u = f_1 u_{f_2}~\text{a.e. in}~\Omega.$$ 

\textbf{Step 2.} Let us assume that $ f = \sum_{i=1}^m f_{1,i} f_{2,i} \in L^2(\omega_1) \otimes L^2(\omega_2)$ where $(f_{1,i}, f_{2,i}) \in L^2(\omega_1) \times L^2(\omega_2)$ for any $i \in \{1,...,m\}$. By linearity we obtain that $$u = \sum_{i=1}^m f_{1,i} u_{f_{2,i}}~\text{a.e. in}~\Omega.$$
Using this identity and the above step we obtain that $ D^2_{X_2} u \in L^2(\Omega)^{ (N-q)^{2}}$, in particular one has, for a.e. $%
X_{1}\in \omega _{1}$  
\begin{equation*}
u(X_{1},\cdot )\in H^{2}(\omega _{2})\text{,}
\end{equation*}%
Now, from (\ref{u-limit-weak-bis}) the elliptic regularity on $\omega _{2}$ shows that there exists $C_{_{\l,A_{22},\omega_2}}>0 $ independent of $X_1$, such that for a.e. $X_{1}\in \omega _{1}$ :%
$$
\| D^2_{X_2} u(X_1,\cdot) \|_{L^2(\omega_2)^{(N-q)^2}} \le C_{_{\l,A_{22},\omega_2}} \| f(X_{1},\cdot) \|_{L^2(\omega_2)}.
$$
We integrate this identity over $\omega_1$ and we obtain
$$
\| D^2_{X_2} u \|_{L^2(\Omega)^{(N-q)^2}} \le C_{_{\l,A_{22},\omega_2}}  \| f \|_{L^2(\Omega)}.
$$
\textbf{Step 3.} Let $f \in L^2(\Omega)$. There exists a sequence $(f_n)_{n \ge 0}$ of functions of $L^2(\omega_1) \otimes L^2(\omega_2)$ converging to $f$ in $L^2(\Omega)$ as $n$ goes to infinity. 
 Let $u_n$ be the unique solution in $%
H_{0}^{1}(\Omega ;\omega _{2})$ of 
\begin{equation} \label{weakstep3}
\int_{\Omega }A_{22}\nabla_{x_2}u_n \cdot \nabla_{x_2}\varphi \diff x=\int_{\Omega }f_n\varphi \diff x\text{,
\ }\forall \varphi \in H_{0}^{1}(\Omega ;\omega _{2}).  
\end{equation}
Subtracting (\ref{weakstep3}) from (\ref{u-limit-weak}) and taking $\varphi = u_n-u$, then by using ($\ref{hypA_221}$) and ($\ref{sob}$) we obtain that
\begin{equation}\label{Conv1-step3}
u_n \longrightarrow u \text{ in } L^{2}(\Omega) \text{ as } n \rightarrow \infty \text{ strongly}.
\end{equation}
From the previous step we obtain that $ D_{X_2}^2 u_n \in L^2(\Omega)^{(N-q)^2} $ for any $n \ge 0$ and 
\begin{equation}\label{D2X2Un}
\| D^2_{X_2} u_n \|_{L^2(\Omega)^{(N-q)^2}} \le C_{\l,A_{22},\omega_2}  \| f_n \|_{L^2(\Omega)},~\text{for any}~n\ge 0.
\end{equation}
Using the fact that the sequence $(f_n)_{n \ge 0}$ is bounded in $L^2(\Omega)$ we obtain that the sequence $(D^2_{X_2} u_n)_{n \ge 0}$ is bounded in $L^2(\Omega)^{(N-q)^2}$. Therefore, there exists a subsequence still labelled $(D_{X_2}^2 u_n)$ such that for any $i,j \in \llbracket q+1,N \rrbracket$ there exists $u_{ij}^{\infty} \in L^2(\Omega)$ such that 
\begin{equation} \label{Conv2-step3}
 \partial ^2 _{ij} u_{n} \rightharpoonup u_{ij}^{\infty} \text{ as } n \rightarrow \infty \text{, weakly in }L^2(\Omega).
\end{equation}
Now, for $\phi \in \mathcal{D}(\Omega)$ and $i,j \in \llbracket q+1,N \rrbracket$ we have 
$$\int_{\Omega} \partial ^2 _{ij}u_n \phi dx = \int_{\Omega} u_n \partial ^2 _{ij} \phi dx$$
Passing to the limit in the above identity by using (\ref{Conv1-step3}) and (\ref{Conv2-step3}) to obtain 
$$\int_{\Omega} u_{ij}^{\infty} \phi dx = \int_{\Omega} u \partial ^2 _{ij} \phi dx \text{, for any }i,j \in \llbracket q+1,N \rrbracket $$
Therefore, we obtain that  $\partial ^2 _{ij}u=u^{\infty}_{ij} \in L^2(\Omega)$ for any $i,j \in \llbracket q+1,N \rrbracket$. Finally, passing to the limit in (\ref{D2X2Un}) we get 
$$ \| D^2_{X_2} u \|_{L^2(\Omega)^{(N-q)^2}} \leq \liminf  \| D^2_{X_2} u_n \|_{L^2(\Omega)^{(N-q)^2}} \leq C_{\l,A_{22},\omega_2} \lim \| f_n \|_{L^2(\Omega)}=C_{\l,A_{22},\omega_2} \| f \|_{L^2(\Omega)}, $$
 and the Lemma \ref{lem-regul1} follows.
\end{proof}
At the next step we prove in the following lemma that $D^2_{X_1 X_2}u\in L^{2}(\Omega )^{q(N-q)}$
\begin{lem}
\label{lem-regul2}Let $\Omega =\omega _{1}\times \omega _{2}$ be an open bounded subset of $\R^N$,  where $\omega
_{1}$ and $\omega _{2}$ are two bounded open subsets of $
\mathbb{R}^{q}$ and $\mathbb{R}^{N-q},$ with $N>q\geq 1$. Assume that $(\ref{hypA2}) $, and $(\ref{hypA_221}) $ are satisfied. Suppose that $f\in L^{2}(\Omega )$ such that $\nabla_{X_1}f\in
L^{2}(\Omega )^q$.
Then: 
\begin{equation*}
\nabla_{X_1}u \in L^2(\Omega)^q~\text{and}~D^2_{X_1 X_2}u\in L^{2}(\Omega )^{q(N-q)},
\end{equation*}%
with: $$\Vert \nabla_{X_1} u\Vert_{L^2(\Omega)^q} \leq C_{_{\lambda,\omega_2}}\Vert \nabla_{X_1}f\Vert_{L^2(\Omega)^q} \text{ and } \Vert D_{X_1X_2}^2 u\Vert_{L^2(\Omega)^{q(N-q)}} \leq C_{{\lambda,\omega_2}}\Vert \nabla_{X_1}f\Vert_{L^2(\Omega)^q}, $$
and for a.e. $X_{1}\in \omega _{1}:\nabla_{X_1}u(X_{1},\cdot )\in H_{0}^{1}(\omega
_{2})$.
\end{lem}
\begin{proof}
We use the difference quotient method of Nirenberg (see for instance \cite{Trudinger}). Let $\omega _{1}^{\prime }\subset \subset \omega _{1},$ and  $0<\eta<\text{dist}($ $%
\omega _{1}^{\prime },\partial \omega _{1}).$ Let $i \in \llbracket 1,q \rrbracket $. For a.e $X_{1}\in $ $\omega
_{1}^{\prime },$ we obtain from (\ref{u-limit-weak-bis}) the following identity
\begin{align*}
\int_{\omega _{2}}A_{22}\left( \frac{\nabla_{X_2}u(X_{1}+\eta e_i,\cdot
)-\nabla_{X_2}u(X_{1},\cdot )}{eta}\right) \cdot \nabla_{X_2}\varphi _{2} d X_{2}= ~~~~~~~~~~~~~~\text{~~~~~~~~~~~~~~~~~~~~}\\
 \int_{\omega
_{2}}\left( \frac{f(X_{1}+ \eta e_i,\cdot )-f(X_{1},\cdot )}{\eta}\right) \varphi
_{2} d X_{2}\text{, \ } \\ 
 \forall \varphi _{2}\in H_{0}^{1}(\omega _{2}),
\end{align*}%
where $e_i$ is the $i$-th element of the canonical basis of $\mathbb{R}^N$. 
Testing with $u(X_{1}+\eta e_i,\cdot )-u(X_{1},\cdot )\in H_{0}^{1}(\omega _{2})$
in the above equality and using (\ref{hypA_221}), (\ref{sob}) we obtain%
\begin{equation*}
\int_{\omega _{2}}\left\vert \frac{\nabla_{X_2}u(X_{1}+ \eta e_i,\cdot )-\nabla_{X_2}u(X_{1},\cdot
)}{\eta}\right\vert ^{2} d X_{2}\leq \frac{C_{\omega _{2}}^{2}}{\lambda ^{2}}%
\int_{\omega _{2}}\left\vert \frac{f(X_{1}+ \eta e_i,\cdot )-f(X_{1},\cdot )}{\eta}%
\right\vert ^{2} d X_{2},
\end{equation*}%
where we have used Poincar\'{e}'s inequality (\ref{sob}).  Integrating
over $\omega _{1}^{\prime }$ yields%
\begin{align}
\int_{\omega _{1}^{\prime }\times \omega _{2}}\left\vert \frac{%
\nabla_{X_2}u(X_{1}+\eta e_i,X_{2})-\nabla_{X_2}u(X_{1},X_{2})}{\eta}\right\vert ^{2} d x \leq \text{~~~~~~~~~~~~~~~~~~~~~~~~~~~~~~~~~~~~~~~~} \notag\\
 \frac{%
C_{\omega _{2}}^{2}}{\lambda ^{2}}\int_{\omega _{1}^{\prime }\times \omega
_{2}}\left\vert \frac{f(X_{1}+\eta e_i,X_{2})-f(X_{1},X_{2})}{\eta}\right\vert ^{2} d x
\notag \\
\leq \frac{C_{\omega _{2}}^{2}}{\lambda ^{2}}\left\Vert \partial_{i}f\right\Vert
_{L^{2}(\Omega )}^{2},  \label{est:gradX1X2u}
\end{align}%
and%
\begin{equation}\label{est:gradX1u}
\int_{\omega _{1}^{\prime }\times \omega _{2}}\left\vert \frac{%
u(X_{1}+\eta e_i,X_{2})-u(X_{1},X_{2})}{\eta}\right\vert ^{2} d x\leq \frac{C_{\omega
_{2}}^{4}}{\lambda ^{2}}\left\Vert \partial_{i}f\right\Vert _{L^{2}(\Omega )}^{2}
\end{equation}
The inequalities ($\ref{est:gradX1X2u}$) and ($\ref{est:gradX1u}$) imply that $ D^2_{X_1 X_2} u \in L^2(\Omega)$ and  $\nabla_{X_1} u \in L^2(\Omega)^q$, with $$
\left\Vert D^2_{X_1 X_2}u\right\Vert _{L^{2}(\Omega )^{q(N-q)}}\leq \frac{C_{\omega _{2}}%
}{\lambda }\left\Vert \nabla_{X_1} f \right\Vert _{L^{2}(\Omega )^q} \text{, and }
 \Vert \nabla_{X_1} u  \Vert_{L^2(\Omega)^q} \leq \frac{C_{\omega
_{2}}^{2}}{\lambda}\left\Vert \nabla_{X_1}f\right\Vert _{L^{2}(\Omega )^q}.$$
Finally, let us show that $\nabla _{X_1}u(X_{1},\cdot )\in H_{0}^{1}(\omega _{2})$
for a.e. $X_{1}\in \omega _{1}.$ Let $i \in \llbracket 1,q \rrbracket$, and let $\omega _{1}^{\prime }\subset \subset
\omega _{1},$ be open, and set $\tau _{\eta}u(X_{1},X_{2})=u(X_{1}+\eta e_i,X_{2}),$ for $%
x \in \omega _{1}^{\prime }\times \omega _{2}$, and for $0<\eta<\text{dist}(\omega_1 ',\partial \omega_1)$. Let  $j \in \llbracket q+1,N \rrbracket$, then one can check
that : 
\begin{equation}
\frac{\partial_{j}\tau _{\eta}u-\partial_{j}u}{\eta}\rightarrow \partial_{ij}^{2}u\text{, and }\frac{%
\tau _{\eta}u-u}{\eta}\rightarrow \partial_{i}u\text{ as }\eta\rightarrow 0\text{ in }%
\mathcal{D}^{\prime }(\omega _{1}^{\prime }\times \omega _{2})\text{ }.
\label{convergenceD'}
\end{equation}
Let $(\eta_{n})$ be a sequence such that, for every $n\in 
\mathbb{N}
,$ $0<\eta_{n}<dist($ $\omega _{1}^{\prime },\partial \omega _{1})$ and $%
\eta_{n}\rightarrow 0.$ The sequences $\left( \frac{\tau _{\eta_{n}}u-u}{\eta_{n}}%
\right) $, $\left( \frac{\partial_{j}(\tau _{\eta_{n}}u-u)}{\eta_{n}}\right) $ are
bounded in $L^{2}(\omega _{1}^{\prime }\times \omega _{2})$ thanks to (\ref{est:gradX1u}) and $(\ref{est:gradX1X2u})$, therefore it follows from (\ref{convergenceD'})
that 
\begin{equation*}
\frac{\partial_{j}(\tau _{\eta_{n}}u-u)}{\eta_{n}}\rightharpoonup \partial_{ij}^{2}u\text{ and }%
\frac{\tau _{\eta_{n}}u-u}{\eta_{n}}\rightharpoonup \partial_{i}u\text{ in }L^{2}(\omega
_{1}^{\prime }\times \omega _{2})\text{ weakly.}
\end{equation*}%
Finally, Mazur's Lemma shows that there exists a sequence $(U_{n})$ of convex
combination of $\{\frac{\tau _{\eta_{n}}u-u}{\eta_{n}}\}_{n\in 
\mathbb{N}
}$ such that $\partial_{j}U_{n}\longrightarrow \partial_{j}(\partial_{i}u)=\partial_{ij}^{2}u$ and $%
U_{n}\longrightarrow \partial_{i}u$ as $n\rightarrow\infty$ strongly in $L^{2}(\omega _{1}^{\prime }\times
\omega _{2})$. Now, since $(U_{n})\in H_{0}^{1}(\omega _{1}^{\prime }\times
\omega _{2};\omega _{2})^{%
\mathbb{N}
}$ and the space $H_{0}^{1}(\omega _{1}^{\prime }\times \omega _{2};\omega
_{2})$ is complete with the norm $\left\Vert \nabla_{X_2}(\cdot )\right\Vert
_{L^{2}(\omega _{1}^{\prime }\times \omega _{2})}$ then we deduce that $%
\partial_{i}u\in H_{0}^{1}(\omega _{1}^{\prime }\times \omega _{2};\omega _{2})$
i.e. for a.e. $X_{1}\in \omega _{1}^{\prime }$, $\partial_{i}u(X_{1},\cdot )\in
H_{0}^{1}(\omega _{2})$. Notice that $\omega _{1}$ could be covered by a
countable family of $\omega _{1}^{\prime }\subset \subset \omega _{1}$,
whence for a.e. $X_{1}\in \omega _{1}$, $\partial_{i}u(X_{1},\cdot )\in
H_{0}^{1}(\omega _{2}).$ and finally we obtain that $\nabla _{X_1}u(X_{1},\cdot )\in H_{0}^{1}(\omega _{2})$
for a.e. $X_{1}\in \omega _{1}.$
\end{proof}
We finish by the following lemma
\begin{lem}
\label{lem-regul3}
Let $\Omega =\omega _{1}\times \omega _{2}$ be an open bounded subset of $\R^N$,  where $\omega
_{1}$ and $\omega _{2}$ are two bounded open subsets of $
\mathbb{R}^{q}$ and $\mathbb{R}^{N-q},$ with $N>q\geq 1$. Assume that $(\ref{hypA2}) $, and $(\ref{hypA_221}) $ are satisfied. Let  $f \in L^2(\Omega)$ such that $\nabla_{X_1}f\in
L^{2}(\Omega )^q,$ and $D_{X_1}^{2}f\in L^{2}(\Omega )^{q^2}$, then: 
\begin{equation*}
D_{X_1}^{2}u\in L^{2}(\Omega )^{q^2}\text{ and }\left\Vert D_{X_1}^{2}u\right\Vert
_{L^{2}(\Omega )^{q^2}}\leq \frac{C_{\omega _{2}}^{2}}{\lambda }\left\Vert
D_{X_1}^{2}f\right\Vert _{L^{2}(\Omega )\textcolor{black}{^{q^2}}}.
\end{equation*}
\end{lem}

\begin{proof}
Let $\varphi _{1}\otimes \varphi _{2}\in H_{0}^{1}(\omega _{1})\otimes
H_{0}^{1}(\omega _{2})$. Let $i \in \{1,...,q\}$, testing with 
$ \partial_{x_i} \varphi_1 \varphi_2$ 
in (\ref{u-limit-weak}) we obtain 
\begin{equation*}
\int_{\Omega }A_{22} \nabla_{X_2}u  \cdot \nabla_{X_2}\varphi_2 \partial_{x_i} \varphi_1 d x=\int_{\Omega }f \partial_{x_i} \varphi_1\varphi_2  dx.
\end{equation*}%
According to Lemma \ref{lem-regul2} we have $\partial_{ij}^{2}u\in L^{2}(\Omega )$ for any $j \in \llbracket 1+q,N \rrbracket$ 
then, by integration by part we get 
\begin{equation*}
\int_{\Omega }A_{22}(X_{2})\nabla_{X2}\partial_{i}u(X_{1},X_{2})\cdot \varphi
_{1}(X_{1})\nabla_{X_2}\varphi (X_{2})dx=\int_{\Omega }\partial_{i}f(X_{1},X_{2})\varphi
_{1}(X_{1})\varphi (X_{2})dx.
\end{equation*}%
and hence, for a.e. $X_{1}\in \omega _{1}$ \ we obtain 
\begin{equation*}
\int_{\omega _{2}}A_{22}(x_{2})\nabla_{X_2} \partial_i u(x_{1},x_{2})\cdot \nabla_{X_2}\varphi
_{2}dX_{2}=\int_{\omega _{2}}\partial_{i}f(X_{1},X_{2})\varphi _{2}dX_{2}\text{, \ }%
\forall \varphi _{2}\in H_{0}^{1}(\omega _{2}).
\end{equation*}%
Repeating the same method as in proof of Lemma \ref{lem-regul2}. Then, for $i' \in \{1,...,q\}$ and
$\omega _{1}^{\prime }\subset \subset \omega _{1},$ and for a.e.$%
X_{1}\in \omega _{1}^{\prime }$ and for $0<\eta<\text{dist}($ $\omega _{1}^{\prime
},\partial \omega _{1})$ we obtain that
\begin{align*}
\int_{\omega _{2}}\left\vert \nabla_{X_2}\left( \frac{%
\partial_{i}u(X_{1}+ \eta e_{i'},X_{2})-\partial_{i}u(X_{1},X_{2})}{\eta}\right) \right\vert
^{2}dX_{2}\leq \text{~~~~~~~~~~~~~~~~~~~~~~~~~~}\\ \frac{C_{\omega _{2}}^{2}}{\lambda ^{2}}\int_{\omega
_{2}}\left\vert \frac{\partial_{i}f(X_{1}+\eta e_{i'},X_{2})-\partial_{i}f(X_{1},X_{2})}{\eta}%
\right\vert ^{2}dX_{2}.
\end{align*}
By the above lemma we have for a.e.$X_{1}\in \omega _{1}^{\prime
}:\partial_{i}u(X_{1}+\eta e_{i'},\cdot )-\partial_{i}u(X_{1},\cdot )\in H_{0}^{1}(\omega _{2}).$
We integrate over $\omega _{1}^{\prime }$ and we apply (\ref{sob}) to obtain
\begin{align*}
\int_{\omega _{1}^{\prime }\times \omega _{2}}\left\vert \frac{%
\partial_{i} u(X_{1}+\eta e_{i'},X_{2})-\partial_{i} u(X_{1},X_{2})}{\eta}\right\vert ^{2}dx\leq   \text{~~~~~~~~~~~~~~~~~~~~~~~~~~~~~~~~~~~~~~~~~~}\\
\frac{C_{\omega _{2}}^{4}}{\lambda ^{2}}\int_{\omega _{1}^{\prime }\times
\omega _{2}}\left\vert \frac{\partial_{i}f(X_{1}+ \eta e_{i'},X_{2})-\partial_{i}f(x_{1},x_{2})}{\eta}%
\right\vert ^{2}dx.
\end{align*}%
Whence, $\partial_{ii'}^{2}u\in L^{2}(\Omega )$ and $\left\Vert
\partial_{ii'}^{2}u\right\Vert _{L^{2}(\Omega )}\leq \frac{C_{\omega _{2}}^{2}}{%
\lambda }\left\Vert \partial_{ii'}^{2}f\right\Vert _{L^{2}(\Omega )}.$
\end{proof}
In conclusion, Theorem \ref{thm-regularityH2-u} follows from 
Lemmas \ref{lem-regul1}, \ref{lem-regul2}, and \ref{lem-regul3}.

\section{The Analysis of the numerical scheme}
\subsection{Numerical scheme and the main result}

In this section, we assume that $N\in\{2,3\}$ and that the computational domain is  $\Omega = (0,1)^N$. \\
Let $M_i \in \mathbb{N}$, $M_i \ge  2$, $i \in \llbracket 1,N \rrbracket$ and let $\{h_{k_{i}}^{x_i} >0 , k_i = 1, . . . , M_i \}$, for $i \in \llbracket 1,N \rrbracket $. such that 
$$
\sum_{k_{i}=1}^{M_i} h_{k_{i}}^{x_i} =1, \text{ for } i \in \llbracket 1,N \rrbracket,
$$
and let us define the step size $h$ of the discretization by:
$$
h = \max_{i \in \llbracket 1,N \rrbracket}(h_{1}^{x_i},...,h_{M_{i}}^{x_i}).
$$
Let $(x^i_{k_i})_{0 \le k_i \le M_i}$, $i \in \llbracket 1,N \rrbracket$  be the families of real numbers defined by
$$
x^i_{k_i} - x^i_{k_i-1} = h_{k_i}^{x_i}, \text{ for } i \in \llbracket 1,N \rrbracket, \text{ and } k_i =1,..,M_i,
$$
with $x_0^i =0$ for $i \in \llbracket 1,N \rrbracket$.
We define a rectangular
mesh $\mathcal{R}_h=(R_{k_1,...,k_N})_{ \substack{ 1 \le k_i \le M_i \\i \in \llbracket 1,N \rrbracket  }}$ on $\Omega $ by letting 
$$
R_{k_1,...,k_N} = \prod_{{i}=1}^N(x^i_{k_{i}-1},x^i_{k_{i}})
$$
We denote by $\mathcal{S}$ the set of the nodes of the mesh that is
$$
\mathcal{S} = \{ (x^i_{k_i})_{1 \leq i \leq N}, k_i \in \{0,...,M_i\}  \}.
$$
We denote $\mathbb{Q}_{1}(K)$ the space of real polynomials in two variables of partial
degree less or equal to $1$ over $K \subset \R^N$. 
We define the finite dimensional spaces $W_h \subset H^1(\Omega) $ by
\begin{equation*}
W_{h}=\left\{ v\in C(\bar{\Omega})\text{, }v_{|R}\in \mathbb{Q}_{1}(R)~\text{for any}~
R\in \mathcal{R}_{h}\right\}.
\end{equation*}%
As usual at the continuous level or for variational discrete formulation (as in the finite element context), the Dirichlet boundary conditions are incorporated in the definition of the discrete space $V_h \subset H_0^1(\Omega) $ defined by
\begin{equation*}
V_{h}=\left\{ v\in W_{h},\text{ and }v=0\text{ on }\partial \Omega \right\}.
\end{equation*}
Mention that the discrete space  $V_h $ can be written as a tensor product that is 
\begin{equation} \label{Vh-tensor}
V_h= \otimes_{i=1}^N V_{h}^i ,
\end{equation}
where, for $i \in \llbracket 1,N \rrbracket$:
\begin{equation*}
V_{h}^i=\left\{ v\in C([0,1]),~ v_{|(x_{k_{i}-1},x_{k_i})}\in \mathbb{P}_{1}((x_{k_{i}-1},x_{k_i})),~k_i \in \{1,...,M_i\},~
v(0)=v(1)=0\right\},
\end{equation*}%
where $\mathbb{P}_{1}(I)$ is the space of real polynomials in one variable of degree
less or equal to $1$ over $I \subset \R$.
Recall the Sobolev embedding $$H^{2}(\Omega )\hookrightarrow C(\bar{\Omega})$$
which holds in dimension 2 and 3 for Lipschitz domains. We define the classical interpolation operator $$%
I_{h}:H^{2}(\Omega )\longrightarrow W_{h}$$ by 
\begin{equation*}
I_{h}(v)(x)=\sum\limits_{s\in \mathcal{S}}v(x)N_{s}(x)
\end{equation*}%
where $(N_{s})_{s\in \mathcal{S}}$ is the nodal basis. There exists $C_{_N}>0$ such that for any $v \in H^2(\Omega)$ \cite{Suzan}:
\begin{equation}\label{interpolestH1}
\| \nabla (v - I_h(v)) \|_{L^2(\Omega)^N} \le C_{_N} h \Big( \sum_{i=1}^N\| \partial_{i}^2 v \|_{L^2(\Omega)}^2  \Big)^{1/2},
\end{equation}
and 
\begin{equation}\label{interpolestL2}
\| v - I_h(v) \|_{L^2(\Omega)^N} \le C_{_N} h \Big( \sum_{i=1}^N\| \partial_{i} v \|_{L^2(\Omega)}^2  \Big)^{1/2},
\end{equation}
The numerical scheme to approximate problem ($\ref{u-perturbed-weak}$) is 
\begin{equation}
\left\{ 
\begin{array}{l}
\int_{\Omega }A_{\epsilon }\nabla u_{\epsilon ,h}\cdot \nabla
v \diff x=\int_{\Omega } I_h(f) v \diff x, \quad \quad \forall v\in V_{h},%
\text{\ \ \ \ \ \ } \\ 
u_{\epsilon ,h}\in V_{h}.\text{\ \ \ \ \ \ \ \ \ \ \ \ \ \ \ \ \ \ \
\ \ \ \ \ \ \ \ \ \ \ \ \ \ \ \ \ \ \ \ \ \ \ \ \ \ \ \ \ \ \ \ \ \ \ \ \ \
\ \ \ \ \ \ \ \ \ \ \ \ \ \ \ \ \ \ \ \ }%
\end{array}%
\right. ,  \label{1num}
\end{equation}
Now, we are ready to give the main theorem of this section
\begin{thm} \label{theorem-num}
Let $\Omega = (0,1)^N$, with $N \in \{2,3\}$. Assume that $A$ satisfies (\ref{hypA1}), (\ref{A-class-C1}), (\ref{Anullbord}), and (\ref{A22}). Let $f$ such that (\ref{fH2}) is satisfied, then there exists a positive constant $C_{_{\l,f,\Omega,A}}$ independent of $h$ and $\epsilon$ such that 
\begin{equation} \label{est-uniform}
\Vert \nabla_{X_{2}}(u_{\epsilon,h}-u_{\epsilon}) \Vert_{L^{2}(\Omega)^{N-q}} \leq C_{_{\l,f,\Omega,A}} h^{\frac{1}{5}}, 
\end{equation}
where $u_{\epsilon,h}$ and $u_{\epsilon})$ are the solutions of (\ref{1num}) and (\ref{u-perturbed-weak}) respectively. Moreover, if we assume, in addition, that $f \in H^1_0(\Omega)$ then we have 
\begin{equation} \label{est-uniform-H0}
\Vert \nabla_{X_{2}}(u_{\epsilon,h}-u_{\epsilon}) \Vert_{L^{2}(\Omega)^{N-q}} \leq C_{_{\l,f,\Omega,A}} h^{\frac{1}{3}}.
\end{equation}
\end{thm}
This theorem follows immediately from \textcolor{black}{estimates} of type (\ref{glob1}) and (\ref{glob2}). Each one of these two \textcolor{black}{estimates} will be the subject of the next subsections.
\subsection{The first \textcolor{black}{estimate} of type (\ref{glob1})} 
In this subsection we prove the following
\begin{prop} \label{prop-num1}
Suppose that $N \in \{2,3\}$. Assume that $A$ satisfies assumptions (\ref{hypA1}), (\ref{A-class-C1}),  and (\ref{Anullbord}). Assume that $f$ satisfies assumption (\ref{fH2}), then there exists $C_{_{\l,f,A,\Omega}}>0$ independent of $\epsilon$ and $h$ such that
\begin{equation*}
\Vert \nabla_{X_2}(u_{\epsilon,h}-u_{\epsilon}) \Vert _{L^2(\Omega)^{N-q}} \leq C_{_{\l,f,A,\Omega}} \frac{h}{\epsilon^2}.
\end{equation*}
\end{prop}

\begin{proof}
Let $w_{\epsilon,h}$ be the solution of the following 
\begin{equation}
\left\{ 
\begin{array}{l}
\int_{\Omega }A_{\epsilon }\nabla w_{\epsilon ,h}\cdot \nabla
v d x=\int_{\Omega } f v d x, \quad \quad \forall v\in V_{h},%
\text{\ \ \ \ \ \ } \\ 
w_{\epsilon ,h}\in V_{h}.\text{\ \ \ \ \ \ \ \ \ \ \ \ \ \ \ \ \ \ \
\ \ \ \ \ \ \ \ \ \ \ \ \ \ \ \ \ \ \ \ \ \ \ \ \ \ \ \ \ \ \ \ \ \ \ \ \ \
\ \ \ \ \ \ \ \ \ \ \ \ \ \ \ \ \ \ \ \ }%
\end{array}%
\right. ,  \label{1numbis}
\end{equation}
By subtracting (\ref{1numbis}) from (\ref{1num}), and by testing by $(u_{\epsilon ,h}-w_{\epsilon ,h})$ to get 
\begin{equation*}
\Vert \nabla_{X_2}(u_{\epsilon ,h}-w_{\epsilon ,h}) \Vert _{L^2(\Omega)^{N-q}} \leq  \frac{C_{\omega_2}}{\lambda} \Vert I_h(f)-f \Vert _{L^2(\Omega)}.
\end{equation*}
Since $f \in H^2(\Omega)$ (thanks to assumption (\ref{fH2})), then by applying (\ref{interpolestL2}) to the second member of the above inequality, we obtain
\begin{equation}
\Vert \nabla_{X_2}(u_{\epsilon ,h}-w_{\epsilon ,h}) \Vert _{L^2(\Omega)^{N-q}} \leq  C_{_{\l,\omega_2,N}} h,
\label{w-u-perturb}
\end{equation} 
notice that $h^2 \leq h$ since $h \in (0,1]$. \\
Now, subtracting (\ref{u-perturbed-weak}) from (\ref{1numbis}) and using the Galerkin orthogonality and (\ref{hypA1}), we obtain for any $v \in V_h$,
\begin{align*}
\lambda \epsilon^2 \int_\Omega \vert \nabla_{X_1}(w_{\epsilon ,h}-u_{\epsilon}) \vert ^2 dx + \lambda \int_\Omega \vert \nabla_{X_2}(w_{\epsilon ,h}-u_{\epsilon}) \vert ^2 dx \leq \int_\Omega A_{\epsilon} \nabla (w_{\epsilon ,h}-u_{\epsilon}) \cdot \nabla (u_{\epsilon}-v) dx.
\end{align*}
Remark that by a direct application of classical Céa's lemma we obtain an \textcolor{black}{estimate} of order $O(\frac{h}{\epsilon^4})$. We will improve that by using the anisotropic nature of the perturbation, so 
let us develop the right hand side of the above inequality to obtain for any $v \in V_h$ the following
\begin{align*}
\lambda \epsilon^2 \int_\Omega \vert \nabla_{X_1}(w_{\epsilon ,h}-u_{\epsilon}) \vert ^2 dx + \lambda \int_\Omega \vert \nabla_{X_2}(w_{\epsilon ,h}-u_{\epsilon}) \vert ^2 dx 
\leq \text{                                                ~~~~~~~~~~~~~~~~~~~~~~~~~~~~~~~~~~~~~~~~}\\ 
\epsilon^2 \int_\Omega A_{11} \nabla_{X_1} (w_{\epsilon ,h}-u_{\epsilon}) \cdot \nabla _{X_1} (u_{\epsilon}-v) dx 
+ \epsilon \int_\Omega A_{12} \nabla_{X_2} (w_{\epsilon ,h}-u_{\epsilon}) \cdot \nabla_{X_1} (u_{\epsilon}-v) dx 
+ \\
 \epsilon \int_\Omega A_{21} \nabla_{X_1} (w_{\epsilon ,h}-u_{\epsilon}) \cdot \nabla_{X_2} (u_{\epsilon}-v) dx 
+ \int_\Omega A_{22} \nabla_{X_2} (w_{\epsilon ,h}-u_{\epsilon}) \cdot \nabla_{X_2} (u_{\epsilon}-v) dx.
\end{align*}
By using boundedness of $A$ (thanks to (\ref{hypA2}) or (\ref{A-class-C1})), and Young's inequality to each term in the right hand side of the previous inequality, we obtain for any $v \in V_h$
\begin{align*}
\frac{\lambda \epsilon^2 }{2}\int_\Omega \vert \nabla_{X_1}(w_{\epsilon ,h}-u_{\epsilon}) \vert ^2 dx + \frac{\lambda}{2} \int_\Omega \vert \nabla_{X_2}(w_{\epsilon ,h}-u_{\epsilon}) \vert ^2 dx 
\leq  C_{_{A,\l}} \int_\Omega \vert \nabla (u_{\epsilon}-v) \vert ^2 dx.
\end{align*}
Now, we take $v=I_h(u_\epsilon)$ (which belongs to $H^1_0 \cap H^2(\Omega)$) in the previous inequality, then we obtain 
\begin{equation*}
\Vert \nabla _{X_2} (w_{\epsilon ,h}-u_{\epsilon}) \Vert _{L^2(\Omega)^{N-q}} \leq C_{_{A,\l}} \Vert \nabla (I_h(u_\epsilon)-u_\epsilon) \Vert _{L^2(\Omega)^N}.
\end{equation*}
Applying (\ref{interpolestH1}) to right hand side of the above inequality to obtain
\begin{equation*}
\Vert \nabla _{X_2} (w_{\epsilon ,h}-u_{\epsilon}) \Vert _{L^2(\Omega)^{N-q}} \leq C_{_{A,\l}} h \Big(\sum_{i=1}^N\Vert \partial_i^2 u_\epsilon \Vert _{L^2(\Omega)}^2 \Big)^{1/2}.
\end{equation*}
Therefore, by applying Theorem \ref{theorem-H2-upertubed} to the right hand side of the previous inequality we obtain
\begin{equation*}
\Vert \nabla _{X_2} (w_{\epsilon ,h}-u_{\epsilon}) \Vert _{L^2(\Omega)^{N-q}} \leq C_{_{\l,f,A,\Omega} }\frac{h}{\epsilon^2}.
\end{equation*}
Finally, we combine the above inequality with (\ref{w-u-perturb}) and we use the triangle equality to obtain the expected result.
\end{proof}
\subsection{The second \textcolor{black}{estimate} of type (\ref{glob2}) } 

The proof of the \textcolor{black}{estimate} is based on the following theoretical result proved in our previous work (see proof of Lemma 3.9 in \cite{ogabi4}), that is the analogous discrete version of the continuous version given in Theorem \ref{thm-tauxconvergence}.
\begin{lem} \label{lem-Galerkin} \cite{ogabi4}
Suppose that assumptions of Theorem \ref{thm-tauxconvergence} hold. Let $V=V_{1}\otimes V_{2}$ where $V_1 \subset H^1_0(\omega_1)$ and $V_2 \subset H^1_0(\omega_2)$ are finite dimension spaces. Let $g \in V$. Let $u_{\epsilon,V,g}$ and $u_{V,g}$ be the solutions of 
\begin{equation*}
\left\{ 
\begin{array}{l}
\int_{\Omega }A_{\epsilon
}\nabla u_{\epsilon ,V,g}\cdot \nabla \varphi dx=\int_{\Omega }g\text{ }%
\varphi dx\text{, }\forall \varphi \in V\text{\ \ \ \ \ } \\ 
u_{\epsilon ,V,g}\in V\text{.}%
\end{array}%
\right.  
\end{equation*}
and
\begin{equation*}
u_{V,g} \in V \text{, }\int_{\Omega }A_{22}\nabla _{X_{2}}u_{V,g} \cdot
\nabla _{X_{2}}\varphi dx=\int_{\Omega }g\text{ }\varphi dx,\text{\ }\forall
\varphi \in V.
\end{equation*}
Then, 
\begin{align*}
\Vert \nabla _{X_{2}}u_{\epsilon ,V,g}-\nabla _{X_{2}}u_{V,g} \Vert
_{L^{2}(\Omega )^{N-q}} \leq  C_{_{\l,A,\Omega}} \Big( \Vert \nabla
_{X_{1}} g \Vert_{L^{2}(\Omega )^{q}}+ \Vert g \Vert
_{L^{2}(\Omega } \Big) \times \epsilon.
\end{align*}%
\end{lem}
\begin{rmk} \label{rmk-tensor}
The space $V_h$ has the same tensor structure of the space $V$ of Lemma \ref{lem-Galerkin}. In fact:\\
- When $N=2$, identity (\ref{Vh-tensor}) is $V_h= V_{h}^1 \otimes V_{h}^2$.\\
- When $N=3$, identity (\ref{Vh-tensor}) is $V_h= V_{h}^1 \otimes (V_{h}^2 \otimes V_{h}^3)= (V_{h}^1 \otimes V_{h}^2) \otimes V_{h}^3$, the first equality corresponds to the case $q=1$, and the second equality corresponds to the case $q=2$.
\end{rmk}
Now, we are ready to prove the following:
\begin{prop} \label{prop-num2}
Let $\Omega = (0,1)^N$, with $N \in \{2,3\}$. Assume that $A$ satisfies (\ref{hypA1}), (\ref{A-class-C1}), and (\ref{A22}). Let $f$ such that (\ref{fH2}) is satisfied, then there exist a positive constant $C_{_{\l,f,A,\Omega}}>0$ such that
\begin{equation*}
\Vert \nabla_{X_2}(u_{\epsilon,h}-u_{\epsilon}) \Vert _{L^2(\Omega)^{N-q}} \leq C_{_{\l,f,A,\Omega}} (h^{1/4}+\epsilon^{1/2}).
\end{equation*}
In particular, if we assume in addition that $f \in H^1_0(\Omega)$ then we have 
\begin{equation*}
\Vert \nabla_{X_2}(u_{\epsilon,h}-u_{\epsilon}) \Vert _{L^2(\Omega)^{N-q}} \leq C_{_{\l,f,A,\Omega}} (h+\epsilon).
\end{equation*}
\end{prop}

We process by several steps. \\
\textbf{Step 1.}
Let $u_h$ be the solution of the following problem
\begin{equation}
\left\{ 
\begin{array}{l}
\int_{\Omega }A_{22 }\nabla_{X_2} u_{h}\cdot \nabla_{X_2}
v \diff x=\int_{\Omega }I_{h}(f)\text{ }v d x,\quad \forall v\in V_{h}%
\text{\ \ \ \ \ \ } \\ 
u_{h}\in V_{h}.\text{\ \ \ \ \ \ \ \ \ \ \ \ \ \ \ \ \ \ \
\ \ \ \ \ \ \ \ \ \ \ \ \ \ \ \ \ \ \ \ \ \ \ \ \ \ \ \ \ \ \ \ \ \ \ \ \ \
\ \ \ \ \ \ \ \ \ \ \ \ \ \ \ \ \ \ \ \ }%
\end{array}%
\right.   \label{1numlim}
\end{equation} 
We have the following 
\begin{lem} \label{lem-step1}
Suppose that assumptions of Proposition \ref{prop-num2} hold then:\\
- For $f \in H^1_0 \cap H^2(\Omega)$ we have 
\begin{equation}\label{ueh-uh-H0}
\left\Vert \nabla_{X_2}(u_{\epsilon ,h}-u_{h})\right\Vert
_{L^{2}(\Omega )^{N-q}}\leq C_{_{\l,f,A,\Omega}}\times \epsilon.
\end{equation}
- For $f \in H^2(\Omega)$ we have 
\begin{equation}\label{ueh-uh-H2}
\left\Vert \nabla_{X_2}(u_{\epsilon ,h}-u_{h})\right\Vert
_{L^{2}(\Omega )^{N-q}}\leq C_{_{\,f,A,\Omega}}(\epsilon+h^{1/3}).
\end{equation}

\end{lem} 
\begin{proof}
1) Suppose that $f\in H_{0}^{1}(\Omega )\cap H^{2}(\Omega )$, then $I_{h}(f)\in
V_{h}$. Therefore, according to Lemma \ref{lem-Galerkin} and Remark \ref{rmk-tensor} with $g$ and $V$ replaced by $%
I_{h}(f)$ and $V_{h}$ respectively, we have 
\begin{equation*}
\left\Vert \nabla_{X_2}(u_{\epsilon ,h}-u_{h})\right\Vert
_{L^{2}(\Omega )^{N-q}}\leq C_{_{\l,A,\Omega}} \left( \left\Vert \nabla_{X_1}I_{h}(f)\right\Vert
_{L^{2}(\Omega )^q}+\left\Vert I_{h}(f)\right\Vert _{L^{2}(\Omega )}\right)\times \epsilon.
\end{equation*}
Since  $%
I_{h}(f) \in H^1_0(\Omega)$, then by using Poincaré's inequality (\ref{sob-bis}) we obtain
\begin{equation}
\label{ueh-uh}
\left\Vert \nabla_{X_2}(u_{\epsilon ,h}-u_{h})\right\Vert
_{L^{2}(\Omega )^{N-q}}\leq C_{_{\l,A,\Omega}} \times \left\Vert \nabla_{X_1}I_{h}(f)\right\Vert
_{L^{2}(\Omega )^q}\times \epsilon.
\end{equation}
In the other hand (\ref{interpolestH1}) gives 
\begin{align}
\left\Vert \nabla_{X_1}I_{h}(f)\right\Vert _{L^{2}(\Omega )^q}\leq \left\Vert
f\right\Vert _{L^{2}(\Omega )}+C_Nh\left\Vert f\right\Vert _{H^{2}(\Omega)}
\leq C_N \Vert f \Vert _{H^{2}(\Omega)},
\label{interpol-Df}
\end{align}
Therefore, from (\ref{ueh-uh}) and (\ref{interpol-Df}) yields
\begin{equation}\label{ueh-uh-bis}
\left\Vert \nabla_{X_2}(u_{\epsilon ,h}-u_{h})\right\Vert
_{L^{2}(\Omega )^{N-q}}\leq C_{_{\l,f,\Omega,A}}\epsilon.
\end{equation}
2) Now, we suppose that $f \in H^2(\Omega)$, from Lemma \ref{lem-decomp-bis} we use the decomposition $f=f_{\delta}^1+f_{\delta}^2$, and let $u_{\epsilon,h}^i$, $u_h^i \text{, }i=1,2$ be the solutions of (\ref{1num}) and (\ref{1numlim}) respectively with $f$ replaced by $f_{\delta}^i$. The linearity of the equations and the operator $I_h$ gives 
\begin{equation}\label{decomp-sol-num}
u_{\epsilon,h}=u_{\epsilon,h}^1+u_{\epsilon,h}^2 \text{ and } u_{h}=u_{h}^1+u_{h}^2
\end{equation}
As in (\ref{ueh-uh}) we have 
\begin{equation}
\label{ueh-uh-decomp1}
\left\Vert \nabla_{X_2}(u_{\epsilon ,h}^1-u_{h}^1)\right\Vert
_{L^{2}(\Omega )^{N-q}}\leq C_{_{\l,\Omega,A}} \left\Vert \nabla_{X_1}I_{h}(f_{\delta}^1)\right\Vert
_{L^{2}(\Omega )^q}\times \epsilon .
\end{equation}
According to (\ref{interpolestH1}) we obtain
\begin{align*}
\left\Vert \nabla_{X_1}I_{h}(f_{\delta}^1)\right\Vert _{L^{2}(\Omega )^q}\leq \left\Vert
f_{\delta}^1\right\Vert _{L^{2}(\Omega )}+C_N h\left\Vert f_{\delta}^1\right\Vert _{H^{2}(\Omega)},
\end{align*}
and therefore, by applying Lemma \ref{lem-decomp-bis} we get 
\begin{align*}
\left\Vert \nabla_{X_1}I_{h}(f_{\delta}^1)\right\Vert _{L^{2}(\Omega )^q}\leq C_{_{f,\Omega}}\Big(1+ h \delta^{-3/2}\Big).
\end{align*}
Combining this with (\ref{ueh-uh-decomp1}) to obtain 
\begin{equation}
\label{ueh-uh-decomp1bis}
\left\Vert \nabla_{X_2}(u_{\epsilon ,h}^1-u_{h}^1)\right\Vert
_{L^{2}(\Omega )^{N-q}}\leq C_{_{\l,f,A,\Omega}}\Big(1+ h \delta^{-3/2}\Big)\times \epsilon .
\end{equation}
In the other hand, testing by $u_{\epsilon,h}^2$ and $u_{h}^2$ is the corresponding formulations (\ref{1num}) and (\ref{1numlim}), with $f$ replaced by $f_{\delta}^2$, we obtain this basic \textcolor{black}{estimate}
\begin{equation}
\label{ueh-uh-decomp2}
\left\Vert \nabla_{X_2}(u_{\epsilon ,h}^2-u_{h}^2)\right\Vert
_{L^{2}(\Omega )^{N-q}}\leq C_{_{\l,\Omega}} \Vert f_{\delta}^2 \Vert _{L^2(\Omega)}.
\end{equation}
Now, applying Lemma \ref{lem-decomp-bis} to the right hand side of (\ref{ueh-uh-decomp2}) we obtain
\begin{equation}
\label{ueh-uh-decomp2bis}
\left\Vert \nabla_{X_2}(u_{\epsilon ,h}^2-u_{h}^2)\right\Vert
_{L^{2}(\Omega )^{N-q}}\leq C_{_{\l,f,\Omega}} \times \delta^{1/2}
\end{equation}
The combination of (\ref{decomp-sol-num}), (\ref{ueh-uh-decomp1bis}) and (\ref{ueh-uh-decomp2bis}) gives, by the triangle inequality, the following 
 \begin{equation}
\left\Vert \nabla_{X_2}(u_{\epsilon ,h}-u_{h})\right\Vert
_{L^{2}(\Omega )^{N-q}}\leq C_{_{\l,f,A,\Omega}}\Big( \epsilon+ h \delta^{-3/2} \epsilon +\delta^{1/2}\Big).
\end{equation}
Finally, since $\delta$ is arbitrary in $(0,1)$, then by setting $\delta=h^{2/3}$ in the previous inequality we get 
\begin{equation}
\left\Vert \nabla_{X_2}(u_{\epsilon ,h}-u_{h})\right\Vert
_{L^{2}(\Omega )^{N-q}}\leq C_{_{\l,f,A,\Omega}}\Big( \epsilon+ h^{1/3}\Big).
\end{equation}
\end{proof} \smallskip
\textbf{Step 2.} 
We denote $w_h$ the solution to following the problem
\begin{equation}
\left\{ 
\begin{array}{l}
\int_{\Omega }A_{22 }\nabla_{X_2} w_{h}\cdot \nabla_{X_2}
v d x=\int_{\Omega }f \text{ }v d x, \quad \forall v\in V_{h}%
\text{\ \ \ \ \ \ } \\ 
w_{h}\in V_{h}.\text{\ \ \ \ \ \ \ \ \ \ \ \ \ \ \ \ \ \ \
\ \ \ \ \ \ \ \ \ \ \ \ \ \ \ \ \ \ \ \ \ \ \ \ \ \ \ \ \ \ \ \ \ \ \ \ \ \
\ \ \ \ \ \ \ \ \ \ \ \ \ \ \ \ \ \ \ \ }%
\end{array}%
\right.  \label{1numlim-bis}
\end{equation} 
We have the following
\begin{lem} \label{lem-step2}
Assume that assumptions of Proposition \ref{prop-num2} hold, then:\\
- If $f \in H^1_0 \cap H^2(\Omega)$, we have :
$$ \Vert \nabla_{X_2}(w_h-u) \Vert_{L^2(\Omega)} \leq C_{_{\l,f,A,\Omega}}h$$
- If $f \in  H^2(\Omega)$, we have:$$ \Vert \nabla_{X_2}(w_h-u) \Vert_{L^2(\Omega)} \leq C_{_{\l,f,A,\Omega}}h^{1/4}$$

\end{lem}
\begin{proof}

1) Suppose that $f \in H^1_0 \cap H^2(\Omega)$:\\
By using the classical Céa's Lemma we obtain from (\ref{u-limit-weak}) and (\ref{1numlim-bis}) the following 
\begin{equation*}
\left\Vert \nabla_{X_2}(w_{h}-u)\right\Vert _{L^{2}(\Omega )^{N-q}}\leq \frac{%
\left\Vert A_{22}\right\Vert _{L^{\infty }}}{\lambda }\inf_{v\in
V_{h}}\left\Vert \nabla_{X_2}(v-u)\right\Vert _{L^{2}(\Omega )}.
\end{equation*} 
Now, according to Theorem \ref{thm-regularityH2-u} and Theorem \ref{thm-tauxconvergence} we have $u \in H_{0}^{1}(\Omega )\cap H^{2}(\Omega )$, then $I_h(u) \in V_h$. Therefore, from the above inequality we obtain
\begin{equation*}
\left\Vert \nabla_{X_2}(w_{h}-u)\right\Vert _{L^{2}(\Omega )^{N-q}}\leq \frac{%
\left\Vert A_{22}\right\Vert _{L^{\infty }}}{\lambda }\left\Vert \nabla_{X_2}(I_h(u)-u)\right\Vert _{L^{2}(\Omega )}
\end{equation*}
Finally, by applying (\ref{interpolestH1}) and Theorem \ref{thm-regularityH2-u} to the right hand side of the above inequality, we get 
\begin{equation*}
\left\Vert \nabla_{X_2}(w_{h}-u)\right\Vert _{L^{2}(\Omega )^{N-q}}\leq C_{_{\l,f,A,\Omega}} h.
\label{wh-u}
\end{equation*}
2)Suppose that $f \in H^2(\Omega)$: \\
We use the decomposition trick of Lemma \ref{lem-decomp-bis}, and let $w_h^i$, $u^i$, $i=1,2$ be the solutions of (\ref{1numlim-bis}) and (\ref{u-limit-weak}) with $f$ replaced by $f_{\delta}^i$ respectively. As in the previous case we have 
\begin{equation*}
\left\Vert \nabla_{X_2}(w_{h}^1-u^1)\right\Vert _{L^{2}(\Omega )^{N-q}}\leq \frac{%
\left\Vert A_{22}\right\Vert _{L^{\infty }}}{\lambda }\left\Vert \nabla_{X_2}(I_h(u^1)-u^1)\right\Vert _{L^{2}(\Omega )},
\end{equation*}
and the application of (\ref{interpolestH1}) to the right hid side of the previous inequality gives
\begin{equation*}
\left\Vert \nabla_{X_2}(w_{h}^1-u^1)\right\Vert _{L^{2}(\Omega )}\leq \frac{%
\left\Vert A_{22}\right\Vert _{L^{\infty }}}{\lambda }\left\Vert u^1\right\Vert _{H^{2}(\Omega )}h,
\end{equation*}
and by Theorem \ref{thm-regularityH2-u} we get 
\begin{equation*}
\left\Vert \nabla_{X_2}(w_{h}^1-u^1)\right\Vert _{L^{2}(\Omega )^{N-q}}\leq C_{_{\l,A,\Omega}}\left\Vert f_{\delta}^1\right\Vert _{H^{2}(\Omega )}h,
\end{equation*}
whence by applying Lemma \ref{lem-decomp-bis} to the right hand side of the previous inequality we obtain 
\begin{equation*}
\left\Vert \nabla_{X_2}(w_{h}^1-u^1)\right\Vert _{L^{2}(\Omega )^{N-q}}\leq C_{_{\l,f,A,\Omega}}\delta^{-3/2}h,
\end{equation*}
Testing with $w_h^2$, $u^2$ in (\ref{1numlim-bis}) and (\ref{u-limit-weak}) (with $f$ replaced by by $f_{\delta}^2$) we get
\begin{equation*}
\left\Vert \nabla_{X_2}(w_{h}^2-u^2)\right\Vert _{L^{2}(\Omega )^{N-q}}\leq C_{_{f,\Omega}}\delta^{1/2}.
\end{equation*}
Combining these two last inequalities and using the triangle inequality to obtain
\begin{equation*}
\left\Vert \nabla_{X_2}(w_{h}-u)\right\Vert _{L^{2}(\Omega )^{N-q}}\leq C_{_{\l,f,A,\Omega}}\Big(\delta^{-3/2}h+\delta^{1/2}\Big).
\end{equation*}
Finally, we choose $\delta=h^{1/2}$ we obtain
\begin{equation*}
\left\Vert \nabla_{X_2}(w_{h}-u)\right\Vert _{L^{2}(\Omega )^{N-q}}\leq C_{_{\l,f,A,\Omega}}h^{1/4}.
\end{equation*}
\end{proof}
\textbf{Step 3.}
We have the following
\begin{lem}\label{lem-step3}
Assume that assumptions of Proposition \ref{prop-num2} hold, then:
$$\left\Vert \nabla_{X_2}(w_{h}-u_h)\right\Vert _{L^{2}(\Omega )}\leq C_{_{f,\Omega}} h.$$
\end{lem}
\begin{proof}
Subtracting (\ref{1numlim-bis}) from (\ref{1numlim}) and testing with $w_h-u_h$, we get 
\begin{equation*}
\left\Vert \nabla_{X_2}(w_{h}-u_h)\right\Vert _{L^{2}(\Omega )^{N-q}}\leq C_{_{\Omega}} \Vert f-I_h(f) \Vert _{L^2(\Omega)}.
\end{equation*}
By using (\ref{interpolestL2}) in the right hand side of the previous inequality we get 
\begin{equation*}
\left\Vert \nabla_{X_2}(w_{h}-u_h)\right\Vert _{L^{2}(\Omega )^{N-q}}\leq C_{_{\Omega}}\Vert f \Vert_{H^1(\Omega)} h.
\label{wh-uh}
\end{equation*}
\end{proof}
\textbf{Step 4. } 
Now, we are ready to conclude. By using  the triangle inequality we get 
\begin{align*}
\left\Vert \nabla_{X_2}(u_{\epsilon,h}-u_\epsilon)\right\Vert _{L^{2}(\Omega )^{N-q}}\leq \left\Vert \nabla_{X_2}(u_{\epsilon,h}-u_h)\right\Vert _{L^{2}(\Omega )^{N-q}} + \left\Vert \nabla_{X_2}(u_{h}-w_h)\right\Vert _{L^{2}(\Omega )^{N-q}}  \\
+ \left\Vert \nabla_{X_2}(w_{h}-u)\right\Vert _{L^{2}(\Omega )^{N-q}} +\left\Vert \nabla_{X_2}(u-u_\epsilon) \right\Vert _{L^{2}(\Omega )^{N-q}}
\end{align*}
Finally, from Lemmas \ref{lem-step1}, \ref{lem-step2}, \ref{lem-step3} and Theorem \ref{thm-tauxconvergence-H1} we get 
\begin{align*}
\left\Vert \nabla_{X_2}(u_{\epsilon,h}-u_\epsilon)\right\Vert _{L^{2}(\Omega )^{N-q}}\leq C_{_{\l,f,A,\Omega}} (\epsilon^{1/2} +h^{1/4}) \text{, when } f \in H^2(\Omega),
\end{align*}
and 
\begin{align*}
\left\Vert \nabla_{X_2}(u_{\epsilon,h}-u_\epsilon)\right\Vert _{L^{2}(\Omega )^{N-q}}\leq C_{_{\l,f,A,\Omega}} (\epsilon +h) \text{, when } f \in H^1_0\cap H^2(\Omega),
\end{align*}
and the proof of Proposition \ref{prop-num2} is achieved.

In conclusion, we combine the \textcolor{black}{estimates} given in Proposition \ref{prop-num1} and Proposition \ref{prop-num2} to obtain 
\begin{align*}
\left\Vert \nabla_{X_2}(u_{\epsilon,h}-u_\epsilon)\right\Vert _{L^{2}(\Omega )}\leq C_{_{\l,f,A,\Omega}}h^{1/5}, \text{ when } f\in H^2(\Omega),
\end{align*}
and 
\begin{align*}
\left\Vert \nabla_{X_2}(u_{\epsilon,h}-u_\epsilon)\right\Vert _{L^{2}(\Omega )}\leq C_{_{\l,f,A,\Omega}} h^{1/3},\text{ when } f\in H^1_0 \cap H^2(\Omega),
\end{align*}
and the proof of Theorem \ref{theorem-num} is finished.
\textcolor{black}{
\begin{rmk} \label{remarkfinal}
Notice that in concrete physical problems $\epsilon$ is very much less-than $1$. In this case, and by using Proposition \ref{prop-num2}, we can write the following sharp estimates:
\begin{equation*}
\Vert \nabla_{X_2}(u_{\epsilon,h}-u_{\epsilon}) \Vert _{L^2(\Omega)^{N-q}} \leq C_{_{\l,f,A,\Omega}} h^{1/4}, \text{ for } f\in H^2(\Omega) \text{ and } \epsilon \leq \sqrt{h}.
\end{equation*} 
\begin{equation*}
\Vert \nabla_{X_2}(u_{\epsilon,h}-u_{\epsilon}) \Vert _{L^2(\Omega)^{N-q}} \leq C_{_{\l,f,A,\Omega}} h, \text{ for } f\in H^2 \cap H^1_0(\Omega) \text{ and } \epsilon \leq h.
\end{equation*} 
\end{rmk}}

\subsection{Numerical experiments}
\textcolor{black}{
We consider the problem (\ref{prototype2d}) in $\Omega=(0,1)^2$ with $f(x)= \sin \pi x_1 \sin \pi x_2$. The exact solution of (\ref{prototype2d}) is given by 
\begin{equation*}
u_\epsilon(x) = \frac{1}{\pi^2(\epsilon^2 +1)}\sin \pi x_1 \sin \pi x_2. 
\end{equation*}
In this example, we place ourselves in the case $f \in H^1_0 \cap H^2 (\Omega)$. \\
In the following table we give the approximation of the error $\left\Vert \partial_{x_2}u_{\epsilon,h}-\partial_{x_2}u_\epsilon\right\Vert _{L^{2}(\Omega )}$ calculated for several values of $h$ and $\epsilon$. \\
}

\begin{table}[htbp] 
\begin{center} 
   \begin{tabular}{|c| c | c | c | c| c| c|}
     \hline
     & $\epsilon=1$ & $\epsilon=0.75$ & $\epsilon=0.5$ & $\epsilon=0.1$ & $\epsilon=0.01$& $\epsilon=10^{-6}$   \\ \hline
    $h=0.1$ & 0.007211 & 0.009230 & 0.011537& 0.014279 & 0.014420 & 0.014422
 \\ \hline
    $h=0.02$ &0.001443 & 0.001847
 & 0.002309 
& 0.002858 & 0.002886 & 0.002886
 \\ \hline
     $h=0.01$ & 0.000721 & 0.000923 
 & 0.001154 
 & 0.001429 & 0.001443 & 0.001443
 \\
     \hline
     $h=0.001$ & $7.21 \times 10^{-5}$ & $ 9.23 \times 10^{-5}$
 & 0.000115
 & 0.000142
 & 0.000144 & 0.000144
 \\ \hline
   \end{tabular}
 \end{center}
   \caption{ \label{table-error}}
 \end{table}

\textcolor{black}{
It is clear that the error is controlled by $Ch^{1/3}$, and that illustrates the result of Theorem \ref{theorem-num}. For the values of $\epsilon$ and $h$ such that $\epsilon \leq h$ we remark that the error is controlled by $Ch$ and that illustrates the sharp estimates given in Remark \ref{remarkfinal}. For the values of $\epsilon$ such that $\frac{1}{\epsilon^2}$ is not big (for instance $\epsilon =1,\text{ }0.75, \text{ or }0.5$) the corresponding values of the error are controlled by $Ch$ and that illustrates the estimate given in Proposition \ref{prop-num1}.\\
Now,let us give the following graphic representation of Table \ref{table-error}
}

\begin{figure}[!ht]
\begin{center}
\includegraphics[scale=0.6]{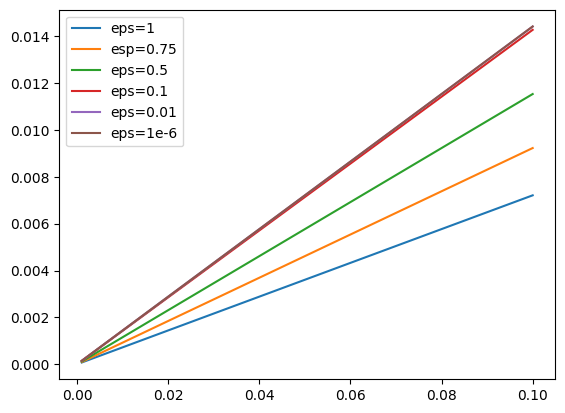}
\caption{ \label{figure-error}}
\end{center}
\end{figure}

\textcolor{black}{
We can see that the error is of the order of $h$ uniformly in $\epsilon$, for this particular example, that is due to the high regularity of $f$ and $A$. \\
}

\end{document}